\documentclass[leqno,12pt]{amsart} 
\usepackage{amssymb}
\usepackage{amsmath}
\usepackage{amsthm}
\usepackage{amscd}
\usepackage{graphicx}
\usepackage[dvips]{color}
\usepackage[all]{xy}
\theoremstyle{plain} 
\newtheorem{theorem}{\indent\sc Theorem}[section]
\newtheorem{lemma}[theorem]{\indent\sc Lemma}
\newtheorem{corollary}[theorem]{\indent\sc Corollary}
\newtheorem{proposition}[theorem]{\indent\sc Proposition}
\newtheorem{claim}[theorem]{\indent\sc Claim}

\theoremstyle{definition} 
\newtheorem{definition}[theorem]{\indent\sc Definition}
\newtheorem{remark}[theorem]{\indent\sc Remark}

\newtheorem{question}[theorem]{\indent\sc Question}
\makeatletter
    
    \@addtoreset{equation}{section}
  \makeatother

\begin{document}

\title[Ueda theory for compact curves with nodes]
{Ueda theory for compact curves with nodes} 

\author[T. Koike]{Takayuki Koike} 

%
\address{
Graduate School of Mathematical Sciences, The University of Tokyo \endgraf
3-8-1 Komaba, Meguro-ku, Tokyo, 153-8914 \endgraf
Japan
}
\email{tkoike@ms.u-tokyo.ac.jp}

\maketitle

\begin{abstract}
Let $C$ be a compact complex curve included in a non-singular complex surface such that the normal bundle is topologically trivial. 
Ueda studied complex analytic properties of a neighborhood of $C$ when $C$ is non-singular or is a rational curve with a node. 
We propose an analogue of Ueda's theory for the case where $C$ admits nodes. 
As an application, we study singular Hermitian metrics with semi-positive curvature on the anti-canonical bundle of the blow-up of the projective plane at nine points in arbitrary position. 
\end{abstract}

\section{Introduction}
Let $X$ be a non-singular complex surface and $C\subset X$ be a compact complex curve. 
Our aim is to investigate complex analytic properties of a neighborhood of $C$ when the normal bundle $N_{C/X}:=j^*\mathcal{O}_X(C)$ is topologically trivial, where we denote by $j$ the inclusion $C\hookrightarrow X$. 
In the present paper, we treat the case where $C$ is a curve with only nodes 
(i.e. $C$ is a $1$-dimensional reduced subvariety of $X$ with only normal crossing singularities. $C$ may be reducible, and may have a self-intersecting curve as an irreducible component). 
One main conclusion of our results in this paper is on the (non-)existence of $C^\infty$ Hermitian metrics with semi-positive curvature on the line bundle $\mathcal{O}_X(C)$ when $C$ is a cycle of rational curves: 

\begin{theorem}\label{thm:maincor}
Let $X$ be a non-singular complex surface and $C\subset X$ be a cycle of rational curves with topologically trivial normal bundle ($C$ is a reduced subvariety of $X$ with only nodes, and may be either a non-singular one or a rational curve with a node).  \\
$(i)$ Assume that $N_{C/X}$ is a flat line bundle (i.e. it can be regarded as an element of $H^1(C, U(1))$, where $U(1):=\{z\in\mathbb{C}\mid |z|=1\}$) with 
$\log d(\mathcal{O}_C, N_{C/X}^{n})=O(\log n)$ as $n\to\infty$, 
where $d$ is the Euclidean distance of $H^1(C, U(1))$. 
Then $\mathcal{O}_X(C)$ admits a $C^\infty$ Hermitian metric with semi-positive curvature. \\
$(ii)$ Assume that $N_{C/X}$ is not flat. 
Then the singular Hermitian metric $|f|^{-2}$ has the mildest singularities among singular Hermitian metrics of $\mathcal{O}_X(C)$ whose local weights are psh (plurisubharmonic), 
where $f\in H^0(X, \mathcal{O}_X(C))$ is a section whose zero divisor is $C$. Epecially, $\mathcal{O}_X(C)$ is nef, however it admits no $C^\infty$ Hermitian metric with semi-positive curvature. 
\end{theorem}

Note that $H^1(C, U(1))$ can be regarded as a subset of $H^1(C, \mathcal{O}_C^*)$, which is the set of all holomorphic line bundles on $C$ (see Remark \ref{rmk:inj_P_0}). 
We can apply Theorem \ref{thm:maincor} to the anti-canonical bundle $K_X^{-1}$ of the blow-up $X$ of the projective plane $\mathbb{P}^2$ at nine points different from each other. 
Let $C_0\subset \mathbb{P}^2$ be a compact curve of degree $3$ including all of the nine points. 
Denote by $C$ the strict transform of $C_0$. 
When $C_0$ is a non-singular elliptic curve, it is known that $K_X^{-1}=\mathcal{O}_X(C)$ admits a $C^\infty$ Hermitian metric with semi-positive curvature if $N_{C/X}$ is an torsion element of $H^1(C, \mathcal{O}_C^*)$ or satisfies the condition as in Theorem \ref{thm:maincor} $(i)$ (It follows from \cite[Theorem 3]{U83}, see \cite{Br}). 
It follows from Theorem \ref{thm:maincor} $(i)$ that the same phenomenon occurs even when $C_0$ is a curve with nodes if $N_{C/X}$ is topologically trivial. 
Thus, even in this case, we can pose the same question as the following Question \ref{q:dem}, which the author learned from Prof. Jean-Pierre Demailly. 

\begin{question}\label{q:dem}
Let $C_0\subset \mathbb{P}^2$ be a non-singular elliptic curve. 
Is there a configuration of nine points $\{p_j\}_{j=1}^9\subset C_0$ such that the anti-canonical bundle $K_X^{-1}$ of the blow-up $X$ of $\mathbb{P}^2$ at $\{p_j\}_{j=1}^9$ admits no $C^\infty$ Hermitian metric with semi-positive curvature? 
\end{question}

Theorem \ref{thm:maincor} $(ii)$ deduces the following: 

\begin{corollary}\label{cor:q_dem}
There exists a configuration of nine points $\{p_j\}_{j=1}^9\subset \mathbb{P}^2$ such that the anti-canonical bundle $K_X^{-1}$ of the blow-up $X$ of $\mathbb{P}^2$ at $\{p_j\}_{j=1}^9$ is nef, however it admits no $C^\infty$ Hermitian metric with semi-positive curvature. 
\end{corollary}

Note that, though Corollary \ref{cor:q_dem} gives an answer to a weak analogue of Question \ref{q:dem},  
we can not give an answer to the original form of Question \ref{q:dem} from Theorem \ref{thm:maincor} $(ii)$. 
It is because $N_{C/X}$ is always flat if $C_0$ is non-singular, where $C$ is the strict transform of $C_0$ (see \cite[Proposition 1]{U83}). 
We also study (singular) Hermitian metrics with semi-positive curvature on $K_X^{-1}$ for arbitrary position of nine points (Theorem \ref{thm:9ptbup}). 

The goal of the present paper 
is to pose an analogue of Ueda theory (\cite{U83}, \cite{U91}) for a curve $C$ with only nodes included in a non-singular surface $X$. 
Prof. Tetsuo Ueda investigated complex analytic properties of a neighborhood of $C$ when $N_{C/X}$ is topologically trivial in the case where $C$ is non-singular (\cite{U83}) 
and the case where $C$ is a rational curve with a node (\cite{U91}). 
When $C$ is a curve with nodes, we define the ``type'' of the pair $(C, X)$ as the supremum of the set of all integers $n$ such that $u_\nu(C, X)=0$ holds for all integer $\nu<n$, 
where $u_\nu(C, X)\in H^1(C, \mathcal{O}_C(N_{C/X}^{-\nu}))$ is the class we will define in \S 3 as an analogue of Ueda's obstruction class posed in \cite{U83}. 
Before describing our main results, we first explain our notations. 
We denote by $\mathcal{P}(C)$ the set of all topologically trivial holomorphic line bundles defined on $C$ and 
by $\mathcal{P}_0(C)$ the set of all flat line bundles defined on $C$. 
We denote by $\mathcal{E}_0(C)$ the set of all torsion elements of $\mathcal{P}_0(C)$, 
and by $\mathcal{E}_1(C)$ the set of all elements $L$ of $\mathcal{P}_0(C)$ which satisfies the condition $\log d(\mathcal{O}_C, L^{n})=O(\log n)$ as $n\to\infty$, 
where $d$ is an invariant distance of $\mathcal{P}_0(C)$ ($\mathcal{E}_1(C)$ does not depend on the choice of $d$, see \cite[\S 4.1]{U83}). 
For $L\in \mathcal{P}(C)$, we denote by $\mathbb{C}(L)$ the sheaf of constant sections of $L$. 
Note that the notion of the constant section is well-defined for $L\in \mathcal{P}(C)$, 
since $L$ admit a flat connection (even when $L\not\in \mathcal{P}_0(C)$, or equivalently, 
even when $L$ admits no flat metric, see Lemma \ref{lem:topologically_trivial_lb}). 
Note also that, the sheaf $\mathbb{C}(L)$ is independent of the choice of the flat connection up to sheaf isomorphism. 

The main results of this paper are the follows. 
The first one is an analogue of \cite[Theorem 3]{U83}: 

\begin{theorem}\label{thm:main}
Let $X$ be a non-singular complex surface, $C$ be a $1$-dimensional reduced compact subvariety of $X$ with only nodes such that $N_{C/X}\in \mathcal{E}_0(C)\cup\mathcal{E}_1(C)$. 
Assume that $i^*N_{C/X}\in \mathcal{E}_0(\widetilde{C})$, where $i\colon\widetilde{C}\to C$ is the normalization of $C$. 
Assume also that $H^1(C, \mathbb{C}(N_{C/X}^{-n}))=0$ holds for each $n\in\mathbb{Z}_{>0}$. 
Then, if the pair $(C, X)$ is of infinite type, then there exists a neighborhood $V$ of $C$ in $X$ such that $\mathcal{O}_V(C)$ is flat. 
\end{theorem}

Next one is an analogue of \cite[Theorem 1, 2]{U83}: 

\begin{theorem}\label{thm:main_finite}
Let $X$ be a non-singular complex surface, $C$ be a $1$-dimensional reduced compact subvariety of $X$ with only nodes such that $G(C)$ is a tree and $N_{C/X}=\mathcal{O}_C$, where $G(C)$ is the dual graph of $C$; i.e. $G(C)$ is the graph such that the vertex set of $G(C)$ is the set of all irreducible components of $C$ and the edge set of $G(C)$ is the set of all nodal points of $C$. 
Assume that the type of the pair $(C, X)$ is a finite number $n\in\mathbb{Z}_{>0}$. 
Assume also that $u_n(C, X)|_{C_\nu}\not=0\in H^1(C_\nu, \mathcal{O}_{C_\nu})$ holds for all irreducible component $C_\nu$ of $C$. 
Then the following holds: \\
$(i)$ For each real number $\lambda>1$, There exists  a neighborhood $V$ of $C$ and a strongly psh function $\Phi_\lambda\colon V\setminus C\to \mathbb{R}$ such that $\Phi_\lambda(p)\to\infty$ and $\Phi_\lambda(p)=O(d(p, C)^{-\lambda n})$ hold as $p\to C$, where $d(p, C)$ is the distance from $p$ to $C$ calculated by using a local Euclidean metric on a neighborhood of a point of $C$ in $V$. \\
$(ii)$ Let $V$ be a neighborhood of $C$ in $X$, $\Psi$ be a psh function defined on$V\setminus C$. 
If there exists a real number $0<\lambda<1$ such that $\Psi(p)=O(d(p, C)^{-\lambda n})$ as $p\to C$, then there exists a neighborhood $V_0$ of $C$ in $V$ such that $\Psi|_{V_0\setminus C}$ is a constant function. 
\end{theorem}

The third one is an generalization of \cite[Theorem 1, 2]{U91}: 

\begin{theorem}\label{thm:main_2}
Let $X$ be a non-singular complex surface, $C$ be a $1$-dimensional reduced compact subvariety of $X$ with only nodes such that the dual graph $G(C)$ is a cycle graph ($G(C)$ may be the graph with one vertex and one edge) and $N_{C/X}\in\mathcal{P}(C)\setminus \mathcal{P}_0(C)$. 
Assume that the type of the pair $(C, X)$ is larger than or equal to $4$. Then the following holds: \\
$(i)$ For each real number $\lambda>1$, There exists  a neighborhood $V$ of $C$ and a strongly psh function $\Phi_\lambda\colon V\setminus C\to \mathbb{R}$ such that $\Phi_\lambda(p)\to\infty$ and $\Phi(p)=O((-\log d(p, C))^{2\lambda})$ hold as $p\to C$, where $d(p, C)$ is the distance from $p$ to $C$ calculated by using a local Euclidean metric on a neighborhood of a point of $C$ in $V$. \\
$(ii)$ Let $V$ be a neighborhood of $C$ in $X$, $\Psi$ be a psh function defined on$V\setminus C$. 
If there exists a real number $0<\lambda<1$ such that $\Psi(p)=O((-\log d(p, C))^{2\lambda})$ as $p\to C$, then there exists a neighborhood $V_0$ of $C$ in $V$ such that $\Psi|_{V_0\setminus C}$ is a constant function. 
\end{theorem}

\cite[Theorem 3]{U83} is shown by using $L^\infty$-norm estimates for $0$-cochains whose coboundary define the obstruction class $u_\nu(C, X)$ for each $\nu$. 
Refining this technique by considering the exterior derivatives of such $0$-cochains, we prove Theorem \ref{thm:main}. 
\cite[Theorem 1, 2]{U83} and \cite[Theorem 1, 2]{U91} are shown by constructing a suitable function $\Phi_\lambda$ on a neighborhood of the curve for each $\lambda>0$. 
We prove Theorem \ref{thm:main_finite} and Theorem \ref{thm:main_2} by generalizing this construction.

The organization of the paper is as follows. 
In \S 2, we study fundamental properties of topologically trivial holomorphic line bundles on a curve with nodes. 
In \S 3, we give the definition of the obstruction class $u_n(C, X)$ and the type of the pair $(C, X)$. 
In \S 4, we prove Theorem \ref{thm:main}. 
In \S 5, we prove Theorem \ref{thm:main_finite}. 
In \S 6, we prove Theorem \ref{thm:main_2}. 
In \S 7, we prove Theorem \ref{thm:maincor}. 
In this section, we also study singular Hermitian metrics with semi-positive curvature on the anti-canonical bundle of the blow-up of $\mathbb{P}^2$ at nine points in arbitrary position. 
\vskip3mm
{\bf Acknowledgment. } 
The author would like to give heartful thanks to 
Prof. Shigeharu Takayama whose comments and suggestions were of inestimable value for my study. 
He also thanks Prof. Tetsuo Ueda, Prof. Yoshinori Gongyo, and Dr. Yusuke Nakamura for helpful comments and warm encouragements. 
He is supported by the Grant-in-Aid for Scientific Research (KAKENHI No.25-2869) and the Grant-in-Aid for JSPS fellows. 

\section{topologically trivial holomorphic line bundles on curves with nodes}

In this section, we study fundamental properties of topologically trivial holomorphic line bundles on a curve with only nodes. 

\begin{lemma}\label{lem:topologically_trivial_lb}
Let $C$ be a compact complex curve with only nodes and $L$ be an element of $\mathcal{P}(C)$. 
Then, in a suitable local trivialization, all the transition functions of $L$ are constant functions valued in $\mathbb{C}^*$: $\mathcal{P}(C)={\rm Image}\left( H^1(C, \mathbb{C}^*)\to H^1(C, \mathcal{O}_C^*)\right)$. 
Especially it follows that $L$ admits a flat connection even when $L\not\in\mathcal{P}_0(C)$. 
\end{lemma}

\begin{proof}
Considering the exponential short exact sequence, it is sufficient to show that the natural map $H^1(C, \mathbb{C})\to H^1(C, \mathcal{O}_C)$ is surjective. 
For simplicity, we show this assertion only when the set $C_{\rm sing}$ of all singular points of $C$ is a unit set $\{p\}\subset C$. 
Fix a sufficiently fine open covering $\{U_j\}_{j=1}^N$ of $C$. 
We may assume that $p\in U_j\implies j=1$. 
Denote by $i\colon \widetilde{C}\to C$ the normalizations of $C$ 
and by $\{\widetilde{U}_\nu\}_{\nu=0}^N$ the open covering of $\widetilde{C}$ such that
\[
i^{-1}(U_j)= \begin{cases}
    \widetilde{U}_0 \sqcup \widetilde{U}_1 & (j=1) \\
    \widetilde{U}_j & (j>1)
  \end{cases}
\]
holds. 
Let $[\{(U_{jk}, f_{jk})\}]$ be an element of $H^1(C, \mathcal{O}_C)$ and set $\{(\widetilde{U}_{jk}, \widetilde{f}_{jk})\}:=i^*\{(U_{jk}, f_{jk})\}$. 
Then $[\{(\widetilde{U}_{jk}, \widetilde{f}_{jk})\}]$ defines an element $[\{(\widetilde{U}_{jk}, \widetilde{f}_{jk})\}]\in H^1(\widetilde{C}, \mathcal{O}_{\widetilde{C}})$. 
Considering the Hodge decomposition on $\widetilde{C}$, it turns out that
$H^1(\widetilde{C}, \mathbb{C})\to H^1(\widetilde{C}, \mathcal{O}_{\widetilde{C}})$ 
is surjective and thus there exists a $0$-cochain $\{(\widetilde{U}_j, \widetilde{F}_j)\}$ and an element $[\{(\widetilde{U}_{jk}, \widetilde{a}_{jk})\}]\in H^1(\widetilde{C}, \mathbb{C})$ such that 
$\{(\widetilde{U}_{jk}, \widetilde{f}_{jk})\}=\{(\widetilde{U}_{jk}, \widetilde{a}_{jk})\}+\delta\{(\widetilde{U}_j, \widetilde{F}_j)\}$
holds. 

Let us consider the case where  $\widetilde{F}_0(\widetilde{p}_0)=\widetilde{F}_1(\widetilde{p}_1)$ holds, 
where $\{\widetilde{p}_\ell\}=i^{-1}(p)\cap \widetilde{U}_\ell$ for $\ell=0, 1$. 
First, in this case, we will show that there exists a $0$-cocycle $\{(U_j, F_j)\}$ such that $i^*\{(U_j, F_j)\}=\{(\widetilde{U}_j, \widetilde{F}_j)\}$. 
As the construction of $F_j$ for $j\not=1$ is trivial, we only explain the construction of $F_1$. 
Let us regard $U_1$ as a neighborhood of $(0, 0)$ in the subset $\{xy=0\}\subset\mathbb{C}^2$, 
$\widetilde{F}_0$ as a function $\widetilde{F}_0(x)$ defined on $\widetilde{U_0}=$ (a neighborhood of $(0, 0)$ in $\{y=0\}$), 
and, $\widetilde{F}_1$ as a function $\widetilde{F}_1(y)$ defined on $\widetilde{U_1}=$ (a neighborhood of $(0, 0)$ in $\{x=0\}$). 
Then we can construct $F_1$ as the function $F_1(x, y):=\widetilde{F}_0(x)+\widetilde{F}_1(y)-\widetilde{F}_0(0)$, which proves the assertion. 
By using this, we obtain that $[\{(U_{jk}, f_{jk})\}]$ and $[\{(U_{jk}, a_{jk})\}]$ coincide with each other as elements of $H^1(C, \mathcal{O}_C)$, where $\{(U_{jk}, a_{jk})\}$ is an element of $H^1(C, \mathbb{C})$ such that $i^*\{(U_{jk}, a_{jk})\}=\{(\widetilde{U}_{jk}, \widetilde{a}_{jk})\}$. This shows the lemma. 

Next we consider thee case where $A:=\widetilde{F}_1(\widetilde{p}_1)-\widetilde{F}_0(\widetilde{p}_0)\not=0$. 
From the same argument as in the previous case, we can construct a $0$-cocycle $\{(U_j, F_j)\}$ such that $i^*\{(U_j, F_j)\}=\{(\widetilde{U}_j, \widetilde{F}_j+\delta_{j0}\cdot A)\}$ holds, where
\[
  \delta_{j0} := \begin{cases}
    1 & (j=0) \\
    0 & (otherwise). 
  \end{cases}
\]
Thus we can show that 
$\{(U_{jk}, f_{jk})\}-\delta\{(U_j, F_j)\}=\{(U_{jk}, a_{jk}+A_{jk})\}$ 
holds, where $\{(U_{jk}, a_{jk})\}$ and $\{(U_{jk}, A_{jk})\}$ are elements of $H^1(C, \mathbb{C})$ such that 
$i^*\{(U_{jk}, a_{jk})\}=\{(\widetilde{U}_{jk}, \widetilde{a}_{jk})\}$ and
$i^*\{(U_{jk}, A_{jk})\}=\{(\widetilde{U}_{jk}, \delta_{0j}\cdot A)\}$ holds. 
This shows the lemma. 
\end{proof}

\begin{remark}\label{rmk:tau_and_a}
Let $C$ be a compact curve with nodes and $i\colon\widetilde{C}\to C$ be the normalization of $C$. 
Let $\{U_j\}$ be an open covering of $C$ which is a finite collection of open sets. 
In the rest of this paper, we always assume that $\{U_j\}$ is fine enough to satisfy the following conditions: 
Each $U_j$ is isomorphic to an open ball in $\mathbb{C}$ or a neighborhood of $(0, 0)$ in $\{xy=0\}\subset\mathbb{C}^2$, and each $U_{jk}$ is isomorphic to an open ball in $\mathbb{C}$. 
Moreover we may assume that, for each $j, k$ such that $U_j\cap C_{\rm sing}\not=\emptyset$ and $U_k\cap C_{\rm sing}\not=\emptyset$, the intersection $U_{jk}$ is the empty set. 
Let $\{\widetilde{U}_\nu\}$ be an open covering of $\widetilde{C}$ which satisfies the following conditions: 
For $j$ such that $U_j\cap C_{\rm sing}=\emptyset$, there exists $\nu$ such that $i^{-1}(U_j)=\widetilde{U}_\nu$. 
For $U_k\cap C_{\rm sing}\not=\emptyset$, there exists $\mu, \lambda$ such that $i^{-1}(U_j)=\widetilde{U}_\mu\cup\widetilde{U}_\lambda$. 
Moreover, each $\widetilde{U}_\nu$ is isomorphic to an open ball in $\mathbb{C}$. 

From the proof of Lemma \ref{lem:topologically_trivial_lb}, it turns out that, for each $L=[\{(U_{jk}, t_{jk})\}]\in H^1(C, \mathbb{C}^*)$, there exists $\tau_{\nu \mu}\in U(1)$ and $a_\nu\in\mathbb{C}^*$ such that 
$t_{jk}=\tau_{\nu\mu}\cdot{a_\mu}/{a_\nu}$ 
holds for each $j, k, \nu, \mu$ with $i^{-1}(U_{jk})=\widetilde{U}_{\nu\mu}$. 
Moreover, we may assume that $a_\nu=1$ holds for each $\nu$ such that $\widetilde{U}_\nu=i^{-1}(U_j)$ for some $U_j$. 

Consider the case where $C$ is a rational curve with a node $p$. 
In this case, we may assume that $\tau_{\nu\mu}=1$ for each $\nu$ and $\mu$, 
and there uniquely exists an open set $U_{k_0}$ such that $p\in U_{k_0}$. 
Then the number $a_{k_0}$ or $a_{k_0}^{-1}$ coincides with the number $\alpha$ appears in \cite{U91}. 
\end{remark}

\begin{remark}\label{rmk:inj_P_0}
Let $C$ be a compact curve with nodes. 
Though the map $H^1(C, \mathbb{C}^*)\to H^1(C, \mathcal{O}_C^*)$ is not injective in general, 
the natural map $H^1(C, U(1))\to H^1(C, \mathcal{O}_C^*)$ is injective. 
For proving this fact, it is sufficient to show that the natural map $H^1(C, \mathbb{R})\to H^1(C, \mathcal{O}_C)$ is injective. 
Let $i\colon\widetilde{C}\to C$ be the normalization and $\widetilde{C}=\bigcup_{\nu=1}^N\widetilde{C}_\nu$ be the irreducible decomposition. 
By considering the short exact sequences as in the proof of Proposition \ref{prop:hodge} below, 
we obtain the following commutative diagram with exact rows: 
\[\xymatrix{
  \bigoplus_{\nu=1}^N\mathbb{R}  \ar[r]\ar[d]^\alpha  & \bigoplus_{p_j\in C_{\rm sing}}\mathbb{R}  \ar[r]\ar[d]^\beta  & H^1(C, \mathbb{R}) \ar[r]\ar[d]  & \bigoplus_{\nu=1}^N H^1(\widetilde{C}_\nu, \mathbb{R}) \ar[d]^\gamma   \\
  \bigoplus_{\nu=1}^N\mathbb{C}  \ar[r]  & \bigoplus_{p_j\in C_{\rm sing}}\mathbb{C}  \ar[r]^\delta  & H^1(C, \mathcal{O}_C) \ar[r]  & \bigoplus_{\nu=1}^N H^1(\widetilde{C}_\nu, \mathcal{O}_{\widetilde{C}_\nu}).
}\]
Taking an element $\xi$ from the kernel of the map $H^1(C, \mathbb{R})\to H^1(C, \mathcal{O}_C)$, we prove that $\xi=0$ holds. 
It follows form the Hodge decomposition theorem that the map $\gamma$ above is an isomorphism. 
Thus there exists an element $\eta\in\bigoplus_{p_j\in C_{\rm sing}}\mathbb{R}$ such that $\eta\mapsto \xi$ holds. 
As $\delta(\beta(\eta))=0$, there exists an element $\zeta\in\bigoplus_{\nu=1}^N\mathbb{C}$ such that $\zeta\mapsto \beta(\eta)$. 
Since it holds that ${\rm Re}\,(\zeta)\mapsto \eta$, we obtain that $\xi=0$ holds. 
\end{remark}

\begin{lemma}\label{lem:tree}
Let $C$ be a compact curve with nodes. 
If the dual graph $G(C)$ of $C$ is a tree, $\mathcal{P}_0(C)=\mathcal{P}(C)$ holds. 
\end{lemma}

\begin{proof}
When the depth of $G(C)$ is equal to $0$, lemma follows from \cite[\S 1.1]{U83}. 
It also turns out that, $\mathcal{P}_0(C)=\mathcal{P}(C)$ holds when there exist $C_1, C_2\subset C$ such that $C=C_1\cup C_2$, $\#(C_1\cap C_2)=1$, and $\mathcal{P}_0(C_j)=\mathcal{P}(C_j)$ holds for $j=0, 1$. 
Thus we can show this lemma by the induction for the depth of $G(C)$. 
\end{proof}

\begin{proposition}\label{prop:hodge}
Let $C$ be a compact complex curve with only nodes and $L$ be an element of $\mathcal{P}(C)$. 
Assume one of the following conditions: \\
$(1)$ $G(C)$ is a tree and $L=\mathcal{O}_C$. \\
$(2)$ Euler number of $G(C)$ is equal to $0$, $L\not=\mathcal{O}_C$, and $i^*L=\mathcal{O}_{\widetilde{C}}$, where $i\colon\widetilde{C}\to C$ is the normalization. \\
Then the natural map $H^1(C, \mathbb{C}(L))\to H^1(C, \mathcal{O}_C(L))\oplus H^1(C, \overline{\mathcal{O}}_C(L))$ is isomorphism, 
where $\mathcal{O}_C(L)$, $\overline{\mathcal{O}}_C(L)$ is the sheaf of holomorphic, anti-holomorphic sections of $L$, respectively. 
\end{proposition}

\begin{proof}
Denote by $\widetilde{C}=\bigcup_{\nu=1}^N \widetilde{C}_\nu$ the irreducible decomposition of $\widetilde{C}$. 
Consider the short exact sequence 
$0\to \mathbb{C}(L)\to i_*\mathbb{C}(i^*L)\to \bigoplus_{p\in C_{\rm sing}}L|_{p}\to 0$, 
where $i_*\mathbb{C}(i^*L)\to \bigoplus_{p\in C_{\rm sing}}L|_{p}$ is the map defined by $(i_*\mathbb{C}(i^*L))_p=L|_p\oplus L|_p\ni (a, b)\mapsto a-b\in L|_p$ for each $p\in C_{\rm sing}$. 
This short exact sequence induces the exact sequence 
\begin{eqnarray}
0&\to& H^0(C, \mathbb{C}(L))\to
\bigoplus_{\nu=1}^N H^0(\widetilde{C}_\nu, \mathbb{C}(i^*L|_{\widetilde{C}_\nu}))
\to \bigoplus_{p\in C_{\rm sing}}L|_p \nonumber \\
&\to& H^1(C, \mathbb{C}(L))\to\bigoplus_{\nu=1}^N H^1(\widetilde{C}_\nu, \mathbb{C}(i^*L|_{\widetilde{C}_\nu}))\to 0. \nonumber
\end{eqnarray}
As $i^*L|_{\widetilde{C}_\nu}=\mathcal{O}_{\widetilde{C}_\nu}$ holds for each $\nu$ under the assumptions $(1), (2)$, we obtain the equation 
\[
{\rm dim}\,H^1(C, \mathbb{C}(L))={\rm dim}\,H^0(C, \mathbb{C}(L))-N+\#(C_{\rm sing})+\sum_{\nu=1}^N{\rm dim}\,H^1(\widetilde{C}_\nu, \mathbb{C}(i^*L|_{\widetilde{C}_\nu})). 
\]
As ${\rm dim}\,H^0(C, \mathbb{C}(L))-N+\#(C_{\rm sing})=0$ holds under the assumptions $(1), (2)$, we obtain 
${\rm dim}\,H^1(C, \mathbb{C}(L))=\sum_{\nu=1}^N{\rm dim}\,H^1(\widetilde{C}_\nu, \mathbb{C}(i^*L|_{\widetilde{C}_\nu}))$. 
We also obtain 
${\rm dim}\,H^1(C, \mathcal{O}_C(L))=
\sum_{\nu=1}^N{\rm dim}\,H^1(\widetilde{C}_\nu, \mathcal{O}_{\widetilde{C}_\nu}(i^*L|_{\widetilde{C}_\nu}))$ and 
${\rm dim}\,H^1(C, \overline{\mathcal{O}}_C(L))=
\sum_{\nu=1}^N{\rm dim}\,H^1(\widetilde{C}_\nu, \overline{\mathcal{O}}_{\widetilde{C}_\nu}(i^*L|_{\widetilde{C}_\nu}))$ from the same argument for the sheaves ${\mathcal{O}}_C(L)$ and $\overline{\mathcal{O}}_C(L)$ instead of $\mathbb{C}(L)$. 
Since each $\widetilde{C}_\nu$ is a compact K\"ahler manifold, 
${\rm dim}\,H^1(\widetilde{C}_\nu, \mathbb{C}(i^*L|_{\widetilde{C}_\nu}))
={\rm dim}\,H^1(\widetilde{C}_\nu, \mathcal{O}_{\widetilde{C}_\nu}(i^*L|_{\widetilde{C}_\nu}))
+{\rm dim}\,
H^1(\widetilde{C}_\nu, \overline{\mathcal{O}}_{\widetilde{C}_\nu}(i^*L|_{\widetilde{C}_\nu}))$ 
holds. 
Thus we obtain the equation 
\[
{\rm dim}\,H^1(C, \mathbb{C}(L))
={\rm dim}\,H^1(C, \mathcal{O}_C(L))
+{\rm dim}\,H^1(C, \overline{\mathcal{O}}_C(L)). 
\]
By the same argument as in \cite[p. 586]{U83}, we can show that the natural map 
$H^1(C, \mathbb{C}(L))\to H^1(C, \mathcal{O}_C(L))\oplus H^1(C, \overline{\mathcal{O}}_C(L))$ is injective. 
Thus this map is isomorphism. 
\end{proof}

\section{Definition of the obstruction classes and the type for the pair $(C, X)$}

Let $X$ be a compact non-singular complex surface and $C\subset X$ be a compact complex curve with only nodes. 
Assume that $N_{C/X}\in \mathcal{P}(C)$. 
In this section, we give the definition of the obstruction class 
$u_n(C, X)\in H^1(C, \mathcal{O}_C(N_{C/X}^{-n}))$ and the type of the pair $(C, X)$ in almost the same manner as in the case where $C$ is non-singular \cite{U83}. 
Let us fix an open covering $\{U_j\}$ of $C$ as in Remark \ref{rmk:tau_and_a}. 
Fix a sufficiently small neighborhood $V$ of $C$ in $X$ and an open covering $\{V_j\}$ of $V$ such that $V_j\cap C=U_j$. 
We may assume that $V_{jk}\not=\emptyset$ only if $U_{jk}\not=\emptyset$. 

\begin{lemma}\label{lem:system_of_order_1}
Let $t_{jk}\in\mathbb{C}^*$ be constants such that $N_{C/X}=[\{(U_{jk}, t_{jk})\}]$. 
Then, by shrinking $V$ and $\{V_j\}$ if necessary, there exists a system $\{(V_j, w_j)\}$ of defining functions $w_j$ of $U_j$ on $V_j$ such that $t_{jk}w_k=w_j+O(w_j^2)$ 
holds on $V_{jk}$ for each $j, k$. 
\end{lemma}

\begin{proof}
Fix a defining functions $v_j$ of $U_j$ on $V_j$ for each $j$. 
Since $[\{(U_{jk}, (v_j/v_k)|_{U_{jk}})\}]=[\{(U_{jk}, t_{jk})\}]\in H^1(\{U_{jk}\}, \mathcal{O}_C^*)$, 
there exists $0$-cochain $\{(U_j, e_j)\}\in\breve{C}^0(\{U_j\}, \mathcal{O}_C^*)$ such that $(v_j/v_k)|_{U_{jk}}=t_{jk}e_j/e_k$. 
By shrinking $V$ and $\{V_j\}$ if necessary, there exists a nowhere vanishing holomorphic function $f_j$ on $V_j$ such that $f_j|_{U_j}=e_j$ for each $j$. 
Then we can show the lemma by considering $w_j:=f_j^{-1}v_j$. 
\end{proof}

Fix $\{t_{jk}\}$, $\{w_j\}$ as in Lemma \ref{lem:system_of_order_1}. 
In the following, we always assume that $V_j\cap C_{\rm sing}=\emptyset$ whenever we consider the set $V_{jk}$. 
Let 
\begin{equation}\label{eq:exp}
t_{jk}w_k = w_j + f^{kj}_2(z_j)\cdot w_j^2 + f^{kj}_3(z_j)\cdot w_j^3 + f^{kj}_4(z_j)\cdot w_j^4 +\cdots 
\end{equation}
be the expansion of $t_{jk}w_k$ by $w_j$ on $V_{jk}$, where $z_j$ is a holomorphic function defined on $V_j$ such that $(z_j, w_j)$ is a coordinates system of $V_j$. 

\begin{definition}
We say that the system $\{w_j\}$ is of {\bf order $n$} if 
the coefficient $f^{jk}_m$ in the expansion (\ref{eq:exp}) is equal to $0$ for all $m\leq n$ on each $V_{jk}$. 
\end{definition}

\begin{proposition}\label{prop:1cocycle_u}
Assume that a system $\{w_j\}$ is of order $n$. 
Then $\{(U_{jk}, f^{kj}_{n+1}|_{U_{jk}})\}$ satisfies the cocycle condition for $\mathcal{O}_X(-nC)|_C$, where $f^{kj}_{n+1}$ is that in the expansion (\ref{eq:exp}). 
\end{proposition}

\begin{proof}
From a simple calculation, we obtain the equation
\[
f^{kj}_{n+1}|_{U_{jk}}=\frac{1}{n}(w_j^{-n}-t_{jk}^{-n}w_k^{-n})|_{U_{jk}}
\]
on each $U_{jk}$, which shows the $1$-cocycle condition. 
\end{proof}

\begin{definition}
We denote by $u_n(C, X)$ the class $[\{(U_{jk}, f^{kj}_{n+1}|_{U_{jk}})\}]\in H^1(C, \mathcal{O}_X(N_{C/X}^{-n}))$. 
We say that the obstruction class $u_n(C, X)$ is {\bf well-defined} if, by shrinking $V$ and $\{V_j\}$ if necessary, there exits a system $\{w_j\}$ of order $n$. 
\end{definition}

\begin{remark}
It can be easily shown that the class $[u_n(C, X)]\in H^1(C, \mathcal{O}_X(N_{C/X}^{-n}))/\mathbb{C}^*$ does not depend on the choice of the system $\{w_j\}$ of order $n$. 
\end{remark}

\begin{proposition}\label{prop:order_u}
Assume that $u_n(C, X)$ is well-defined. 
Then, for each $m< n$, $u_m(C, X)$ is also well-defined and $u_m(C, X)=0$ holds. 
$u_n(C, X)=0$ holds if and only if $u_{n+1}(C, X)$ is well-defined. 
\end{proposition}

\begin{proof}
The first assertion is clear by definition of $u_m(C, X)$. 
Thus we here show the second assertion. 
Assuming $u_n(C, X)=0$, we will construct a new system $\{v_j\}$ of order $n+1$ by shrinking $V$ and $\{V_j\}$ if necessary. 

As $u_n(C, X)=0$ holds, there exists a $0$-cochain $\{(U_j, F_j)\}\in \breve{C}^0(\{U_j\}, \mathcal{O}_C(N_{C/X}^{-n}))$ such that 
\[
F_j-t_{jk}^{-n}F_k=(w_j^{-n}-t_{jk}^{-n}w_k^{-n})|_{U_{jk}}
\]
holds for each $j, k$. 
Extending each $F_j$ to a holomorphic function defined on $V_j$ by shrinking $V$ and $\{V_j\}$ if necessary, we define a new system $\{v_j\}$ by 
$v_j:=w_j\cdot(1-F_jw_j^n)^{-\frac{1}{n}}=w_j+\frac{1}{n}F_jw_j^{n+1}+\cdots$. 
Since $v_j=w_j+O(w_j^{n+1})$, it is clear that our new system $\{v_j\}$ is also of order $n$. 
Moreover, the equation
\[
(v_j^{-n}-t_{jk}^{-n}v_k^{-n})|_{U_{jk}}
=\left.\left((w_j-F_j)^{-n}-t_{jk}^{-n}(w_k^{-n}-F_k)\right)\right|_{U_{jk}}\equiv 0
\]
shows that the system $\{v_j\}$ is of order $n+1$ (see the proof of Proposition \ref{prop:1cocycle_u}). 
\end{proof}

By Proposition \ref{prop:order_u}, we can define the type of the pair $(C, X)$ as follows: 

\begin{definition}\label{def:order_u}
We say that $(C, X)$ is of {\bf type $n\in\mathbb{Z}_{>0}$} if $u_n(C, X)$ is well-defined and $u_n(C, X)\not=0$ holds. 
We say that $(C, X)$ is of {\bf infinite type} if $u_n(C, X)$ is well-defined and $u_n(C, X)=0$ holds for all $n\in\mathbb{Z}_{>0}$. 
\end{definition}

\section{Proof of Theorem \ref{thm:main}}

\subsection{Preliminary for the proof of Theorem \ref{thm:main}}

Let $C$ be a compact complex curve with only nodes and $i\colon\widetilde{C}\to C$ be the normalization. 
Fix an open covering $\{U_j\}$ of $C$ and $\{\widetilde{U}_\nu\}$ of $\widetilde{C}$ as in Remark \ref{rmk:tau_and_a}. 
In the proof of Theorem \ref{thm:main}, as in \cite[p. 601]{U83}, we use the function 
\[
d(\mathcal{O}_C, L):=\inf_{\{t_{jk}\}} \max_{j, k} |1-t_{jk}|. 
\]
as a distance between $\mathcal{O}_C$ and $L\in\mathcal{P}_0(C)$, where the infimum is taken over the set of all $\{t_{jk}\}\subset U(1)$ such that $L=[\{(U_{jk}, t_{jk})\}]\in H^1(C, U(1))$. 
We use the following lemmata for proving the theorem: 

\begin{lemma}\label{lem:KS_const}
For each $L\in \mathcal{P}_0(C)$, there exists a constant $K=K(L)>0$ such that the following  holds: 
For each $1$-cocycle $\alpha=\{(U_{jk}, \alpha_{jk})\}\in Z^1(\{U_j\}, \mathbb{C}(L))$ with $[\alpha]=0\in H^1(\{U_j\}, \mathbb{C}(L))$, There exists a $0$-cochain $\beta=\{(U_j, \beta_j)\}\in \breve{C}^0(\{U_j\}, \mathbb{C}(L))$ such that
$\delta \beta=\alpha$ and $\| \beta\|\leq K\|\alpha\|$ holds, where $\|\alpha\|=\max_{jk}\sup_{U_{jk}}|\alpha_{jk}|,\ \|\beta\|=\max_j\sup_{U_j}|\beta_j|$. 
\end{lemma}

Lemma \ref{lem:KS_const} can be shown in the same manner as the proof of \cite[Lemma 2]{KS}. 
Next lemma is an analogue of \cite[Lemma 4]{U83}. 

\begin{lemma}\label{lem:U83_lem4_const}
There exists a constant $K=K(C)>0$ such that, 
for each $L\in\mathcal{P}_0(C)$ and 
$\beta=\{(U_j, \beta_j)\}\in \breve{C}^0(\{U_j\}, \mathbb{C}(L))$, $d(\mathcal{O}_C, L)\|\beta\|\leq K\|\delta\beta\|$ holds. 
\qed
\end{lemma}

\begin{proof}
Let $G$ be the dual graph of the open covering $\{U_j\}$: 
i.e. the vertex set is the set $\{U_j\}$, 
and there exist an edge connecting $U_j$ and $U_k$ if and only if $U_{jk}\not=\emptyset$. 
Denote by $K_0$ the length of a longest path of $G$. 
We will show the inequality $d(\mathcal{O}_C, L)\leq (1+2K_0)\|\delta\beta\|$ for each $L\in\mathcal{P}_0(C)$ and 
$\beta=\{(U_j, \beta_j)\}$ such that $\|\beta\|=1$. 

Fix $j_0$ such that $|\beta_{j_0}|=1$. 
Then, for each $U_j$, there exists a path $\ell$ of $G$ from $U_{j_0}$ to $U_j$ whose length is $n\leq K_0$. 
Let $U_{\ell_0}, U_{\ell_1}, \cdots, U_{\ell_n}$ be the sequence of open sets corresponding to the path $\ell$ 
($n\leq K$, $\ell_0=j_0, \ell_n=j, U_{\ell_\nu}\cap U_{\ell_{\nu+1}}\not=\emptyset$ for each $\nu<n$). 
First, we will show that $1-|\beta_{\ell_\nu}|\leq \nu\|\delta\beta\|$ holds for each $\nu\leq n$. 
As it is clear in the case $\nu=0$, it is sufficient to show $1-|\beta_{\ell_{\nu+1}}|\leq (\nu+1)\|\delta\beta\|$ assuming that 
$1-|\beta_{\ell_\nu}|\leq \nu\|\delta\beta\|$ holds. 
Since 
$|\beta_{\ell_\nu}|\leq |\beta_{\ell_{\nu+1}}|+|\beta_{\ell_{\nu+1}}-t_{\ell_{\nu+1}\ell_\nu}\beta_{\ell_\nu}|\leq |\beta_{\ell_{\nu+1}}|+\|\delta\beta\|$, 
the assertion holds. 

Thus we obtain the inequality $1-|\beta_j|\leq K_0\|\delta\beta\|$ for all $j$. 
Therefore 
\begin{eqnarray}
d(\mathcal{O}_C, L)
&\leq&\max_{j, k}\left|\frac{\beta_j}{|\beta_j|}-t_{jk}\frac{\beta_j}{|\beta_j|}\right|\nonumber \\
&\leq&\max_{j, k}\left(\left|\frac{\beta_j}{|\beta_j|}-\beta_j\right|+
\left|\beta_j-t_{jk}\beta_k\right|+
\left|\frac{\beta_j}{|\beta_j|}-\beta_k\right|\right) \nonumber \\
&\leq& \|\delta\beta\|+2\max_j\ (1-|\beta_j|) 
\leq (1+2K_0)\|\delta\beta\| \nonumber 
\end{eqnarray}
holds, which shows the lemma. 
\end{proof}

\subsection{Proof of Theorem \ref{thm:main}}
Let $V$ be a sufficiently small neighborhood of $C$ in $X$. 
Fix an open covering $\{V_j\}$ of $V$ such that $U_j=V_j\cap C$, where $\{U_j\}$ is that in the previous subsection. 
Fix also a defining function $w_j$ of $U_j$ in $V_j$ as in Lemma \ref{lem:system_of_order_1}. 
We will prove Theorem \ref{thm:main} by constructing a new system $\{(V_j, u_j)\}$ of defining functions $u_j$ of $U_j$ in $V_j$ such that $t_{jk}u_k=u_j$ 
holds on each $V_{jk}$. 
Just as the proof of \cite[Theorem 3]{U83}, we will construct such a new system $\{(V_j, u_j)\}$ by solving the  functional equation 
\begin{equation}\label{eq:funceq}
w_j=u_j+\sum_{n=2}^\infty F^j_nu_j^n
\end{equation}
on each $V_j$ after defining a suitable holomorphic function $F_j$ on each $V_j$. 

\subsubsection{Construction of $\{F^j_n\}$}

Fix a holomorphic function $z_j$ defined on $V_j$ such that $(z_j, w_j)$ is a coordinates system of $V_j$ for each $j$ such that $V_j\cap C_{\rm sing}=\emptyset$. 
In this subsection, we show the existence of $\{F^j_\nu\}$ such that the following {\bf (Property)$_n$} holds for each integer $n\geq 2$: 

\begin{description}
\item[(Property)$_n$]  
Each $F_j$ depends only on the variable $z_j$ if $V_j\cap C_{\rm sing}=\emptyset$, 
and the system $\{(V_j, u_j)\}$ is of order $n$ if $u_j$ is a defining function of $U_j$ in $V_j$ which satisfies the equation $w_j=u_j+\sum_{\nu=2}^n F^j_\nu u_j^\nu$. 
\qed
\end{description}

We show the existence of such $\{F^j_\nu\}$ by induction for $n$. 
Assuming that there exists a system $\{F^j_\nu\}$ for $\nu=2, 3, \dots, n-1$ which satisfies (Property)$_{n-1}$, we will construct $\{F^j_n\}$ (The construction of $\{F^j_2\}$ is done in the same manner as follows). 

Consider the expansion (\ref{eq:exp}) of $t_{jk}w_k$ by $w_j$ on $V_{jk}$:
\[
t_{jk}w_k = w_j + f^{kj}_2(z_j)\cdot w_j^2 + f^{kj}_3(z_j)\cdot w_j^3 + f^{kj}_4(z_j)\cdot w_j^4 +\cdots
\]
(Recall that we are assuming that $V_j\cap C_{\rm sing}=\emptyset$ whenever we consider the set $V_{jk}$). 
Let 
\begin{equation}
F^k_\nu(w_j, z_j)=\pi_j^*\left(F^k_\nu|_{U_{jk}}\right)+\sum_{\lambda=1}^\infty F^{kj}_{\nu\lambda}(z_j)\cdot w_j^\lambda
\end{equation}
be the expansion of $F^k_\nu$ by the variable $w_j$ on $V_{jk}$, where $\pi_j\colon V_j\to U_j$ is the projection defined by $(w_j, z_j)\mapsto (0, z_j)$. 
Define the system $\{P^{kj}_\nu\}$ by
\begin{equation}\label{eq:def_P}
\sum_{\nu=2}^{n}f^{kj}_\nu(z_j)\cdot \left(w_j+\sum_{\lambda=2}^{n-1} F^j_\lambda(z_j)\cdot w_j^\lambda\right)^\nu
=\sum_{\nu=2}^{n}P^{kj}_\nu(z_j)\cdot w_j^\nu+O(w_j^{n+1}), 
\end{equation}
and the system $\{Q^{kj}_\nu\}$ by
\begin{equation}\label{eq:def_Q}
\sum_{\nu=2}^{n-1} t_{jk}^{-\nu+1}\sum_{\lambda=1}^\infty F^{kj}_{\nu\lambda}(z_j)\cdot \left(w_j+\sum_{\lambda=2}^{n-1} F^j_\lambda(z_j)\cdot w_j^\lambda\right)^\lambda\hskip-2mm\cdot w_j^\nu
=\sum_{\nu=2}^{n}Q^{kj}_\nu(z_j)\cdot w_j^\nu+O(w_j^{n+1}). 
\end{equation}
Note that each $P^{kj}_\nu$ and $Q^{kj}_\nu$ is uniquely determined after determining $\{F^j_\mu\}$ for $\mu<\nu$. 
Calculations by using (Property)$_{n-1}$ shows that the class of $H^1(C, \mathcal{O}_C(N_{C/X}^{-n+1}))$ defined by $(P^{kj}_\nu-t_{jk}^{-n+1}Q^{kj}_{\nu}) |_{U_{jk}}$ coincides with the obstruction class $u_{n-1}(C, X)$ (see \cite[\S 4.2]{U83}). 
Since $u_{n-1}(C, X)=0$, there exists a $0$-cochain $\{(U_j, f^j_n)\}$ such that 
$f^j_n-t_{jk}^{-n+1}f^k_n =(P^{kj}_\nu-t_{jk}^{-n+1}Q^{kj}_{\nu}) |_{U_{jk}}$ holds for each $j, k$. 
The calculation in \cite[\S 4.2]{U83} shows that a system $\{F^j_n\}$ satisfies (Property)$_n$ if 
$F^j_n(w_j, z_j)=f^j_n(z_j)$ for each $j$ such that $V_j\cap C_{\rm sing}=\emptyset$, and 
$F^k_n|_{U_k}=f^k_n$ holds for each $k$. 
\qed

\subsubsection{Construction of $\{u_j\}$} 

In this subsection, we prove Theorem \ref{thm:main} assuming that there exists a system 
$\{F^j_n\}_{n=2}^\infty$ which satisfies (Property)$_n$ for each $n\geq 2$ and the formal power series $X+\sum_{n=2}^\infty (\sup_{V_j}|F^j_n|)\cdot X^n$ has a positive radius of convergence 
(We will prove this assertion from the next subsection). 
For each $j$ such that $U_j\cap C_{\rm sing}=\emptyset$, we define $u_j$ by the functional equation (\ref{eq:funceq}) (It is clear that there uniquely exists a solution of the functional equation (\ref{eq:funceq}). 
Next, let us consider on $U_k$ such that $U_k\cap C_{\rm sing}=\emptyset$. 
Let $(x_k, y_k)$ be a coordinates system of $U_k$ such that $w_k=x_k\cdot y_k$. 
Define new coordinates system $(\widetilde{x}_k, \widetilde{y}_k)$ of $U_k$ by $\widetilde{x}_k=x_k$ and 
$y_k=\widetilde{y}_k+\sum_{n=2}^\infty F^k_nx_k^{n-1}\widetilde{y}_k^n$. 
Then clearly the function $u_k:=\widetilde{x}_k\cdot \widetilde{y}_k$ is a defining function of $U_k$ in $V_k$ and is a solution of the functional equation (\ref{eq:funceq}). 
From (Property)$_n$ for each $n\geq 2$, we can conclude that $t_{jk}u_k=u_j$ 
holds on each $V_{jk}$, 
which proves the theorem. 
\qed

\subsubsection{Convergence of the functional equation (\ref{eq:funceq}) for the case $N_{C/X}\in\mathcal{E}_0(C)$}\label{section:case_torsion}

In this subsection, we prove that the formal power series $X+\sum_{n=2}^\infty (\sup_{V_j}|F^j_n|)\cdot X^n$ has a positive radius of convergence for a suitable choice of $\{F^j_n\}$ when $N_{C/X}\in\mathcal{E}_0(C)$. 

First, we fix a point $p_j\in U_j$ as follows: 
For each $j$ such that $U_j\cap C_{\rm sing}=\emptyset$, we can regard $U_j$ as the ball $\{|z|<\varepsilon_j\}\subset\mathbb{C}$ for some positive number $\varepsilon_j$ via the coordinate function $z_j$. For such $j$, we define $p_j$ by $z_j(p_j)=0\in\{|z|<\varepsilon_j\}$. 
For each $k$ such that $U_k\cap C_{\rm sing}\not=\emptyset$, via a suitable coordinates system $(x_k, y_k)$ of $V_k$, $U_k$ is isomorphic to $\{x_k\cdot y_k=0, |x_k|<\varepsilon_k, |y_k|<\varepsilon_k \}\subset\mathbb{C}^2$ for some positive number $\varepsilon_k$. 
For such $k$, we define $p_k$ by $(x_k(p_k), y_k(p_k))=(0, 0)\in\{x_k\cdot y_k=0, |x_k|<\varepsilon_k, |y_k|<\varepsilon_k \}$. 
We denote by $C^j_n$ the constant $F^j_n(p_j)$ for each $j$ and $n$. 
Note that we may assume that $\varepsilon_j, \varepsilon_k<1$. 
Fix an open covering $\{U_j^*\}$ of $C$ such that $\overline{U_j^*}\subset U_j$ and 
$U_{jk}=\emptyset\ \Rightarrow U^*_{jk}=\emptyset$ holds for each $j, k$. 
Fix also an open neighborhood $V_j^*$ of $U_j^*$ in $V_j$ such that $\overline{V_j^*}\subset V_j$. 
Note that, we may assume that  the coordinate function $z_j\ (x_k, y_k)$ can be extended to the set $V_j\cup V_l^*$ ($V_k\cup V_l^*$) if $U_{jl}\not=\emptyset$ ($U_{kl}\not=\emptyset$, respectively), since $\{U_j\}$ is sufficiently fine. 
Fix a sufficiently large real number $C_0$ such that 
\[
\sup_{V_j\cap V_k^*}\left|\frac{\partial z_k}{\partial z_j}\right|\leq C_0
\]
holds for each $j, k$ with $U_k\cap C_{\rm sing}=\emptyset$, and
\[
\sup_{V_j\cap V_k^*}\left|\frac{\partial x_k}{\partial z_j}\right|\leq C_0,\ 
\sup_{V_j\cap V_k^*}\left|\frac{\partial y_k}{\partial z_j}\right|\leq C_0
\]
holds for each $j, k$ with $U_k\cap C_{\rm sing}\not=\emptyset$. 
Denote by $A(X)=X+\sum_{n=2}^\infty A_nX^n\in\mathbb{C}\{X\}$ the solution of the functional equation
\[
A(X)-X=\frac{M_0(A(X))^2}{1-R_0A(X)} 
\]
with $A'(X)=1$, where $M_0, R_0$ are sufficiently large real number (There uniquely exists such $A(X)$ for each $M_0, R_0$, and clearly has a positive radius of convergence). 
Note that $A_n\geq 0$ holds for each $n\geq 2$. 
In the rest of this subsection, we will prove that
\begin{equation}\label{eq:F_all}
\sup_{V_j}\left|F^j_\nu\right|\leq 3A_\nu
\end{equation}
holds for each $\nu\geq 2$, sufficiently large $M_0, R_0$, and suitable $\{F^j_n\}$. 
For proving this, it is sufficient to show the following claim: 

\begin{claim}\label{claim:F}
For sufficiently large $M_0, R_0$, 
there exists a system $\{F^j_n\}_{n=2}^\infty$ such that {\rm (Property)}$_n$ holds and that, 
for each $k$ such that $V_k\cap C_{\rm sing}\not=\emptyset$, 
there exists a function $G^{k}_{+n}(x_k)$ with $G^{k}_{+n}(x_k(p_k))=G^{k}_{+n}(0)=0$ and a function $G^{k}_{-n}(y_k)$ with $G^{k}_{-n}(y_k(p_k))=G^{k}_{-n}(0)=0$ such that $F^k_n(x_k, y_k)=C^k_n+G^{k}_{+n}(x_k)+G^{k}_{-n}(y_k)$ holds on $V_k$. 
Moreover, 
\begin{eqnarray}\label{eq:F}
\left|C^j_\nu\right|\leq A_\nu\ \text{and}\ 
\left|\frac{dF^j_\nu}{dz_j}\right|\leq A_\nu\ & &\text{if}\ U_j\cap C_{\rm sing}=\emptyset\\
\left|C^k_\nu\right|\leq A_\nu\ \text{and}\ 
\left|\frac{d G^k_{+\nu}}{d x_k}\right|,\ \left|\frac{dG^j_{-\nu}}{d y_k}\right|\leq A_\nu & &\text{if}\ U_k\cap C_{\rm sing}\not=\emptyset\nonumber
\end{eqnarray}
holds for each $\nu$. 
\end{claim}

\begin{remark}
Claim \ref{claim:F} actually deduces inequality (\ref{eq:F_all}). 
Indeed, for example, we can calculate 
\[
\left|G^k_{\pm\nu}(p)\right|=\left|\int_{p_k}^p\left(G^k_{\pm\nu}\right)'\right|\leq\varepsilon_k\sup\left|\left(G^k_{\pm\nu}\right)'\right|
\leq A_\nu 
\]
on each $V_k$ with $V_k\cap C_{\rm sing}\not=\emptyset$ (Recall that we are assuming that $\varepsilon_k<1$). 
Thus we obtain the inequality 
$|F^k_\nu|\leq |C^k_\nu|+|G^{k}_{+\nu}|+|G^{k}_{-\nu}|\leq 3A_\nu$. 
\end{remark}

For each $j$ such that $V_j\cap C_{\rm sing}=\emptyset$, we denote by $G^j_n=G^j_n(z_j)$ the function $F^j_n-C^j_n$. 
From now on, we will prove Claim \ref{claim:F} by induction. 
As $A_2=M_0$, Claim \ref{claim:F} for $\nu=2$ is clear. 
Assuming the existence of $\{F^j_\nu\}_{\nu=2}^{n-1}$ satisfying (Property)$_\nu$ and the inequality (\ref{eq:F}) for $\nu< n$, 
we will construct $\{F^j_{n}\}$ such that (Property)$_n$ and the inequality (\ref{eq:F}) for $\nu=n$ holds. 
First, let us fix sufficiently large real numbers $M_1$ and $R$ such that
\[
\sup_{U_{jk}}\left|f^{kj}_\nu\right|\leq M_1R^\nu,\ \sup_{U_{jk}}\left|\frac{d f^{kj}_\nu}{d z_j}\right|\leq M_1R^\nu
\]
holds for each $j, k$. Then we can show the following lemma: 

\begin{lemma}\label{lem:PQ}
Assume that $\{F^j_\nu\}_{\nu=2}^{n-1}$ satisfies (Property)$_\nu$ and the inequality (\ref{eq:F}) for $\nu< n$. 
Then, by enlarging $R$ if necessary, it holds that 
\[
\sup_{U_{jk}}\left|P^{kj}_n-t_{jk}^{-n+1}Q^{kj}_n\right|\leq \text{the coeff. of}\ X^n\ \text{in}\ \frac{M_2(A(X))^2}{1-6RA(X)}, 
\]
where $M_2:=18R(1+M_1R)$ and ``coeff.'' stands for ``coefficient''. 
\end{lemma}

\begin{proof}
Assuming that $R^{-1}$ is much smaller than the diameters of $V_j$'s, 
we obtain 
\begin{eqnarray}
\left|P^{kj}_n(p)-t_{jk}^{-n+1}Q^{kj}_n(p)\right|&\leq& 
\text{the coeff. of}\ X^n\ \text{in}\ 2R(M_1R+1)\frac{(3A(X))^2}{1-3RA(X)} \nonumber 
\end{eqnarray}
for each $p\in U_{jk}$ by the same argument as in \cite[p. 599]{U83}, which proves the lemma. 
\end{proof}

\begin{lemma}\label{lem:PQ_derivative}
Assume that $\{F^j_\nu\}_{\nu=2}^{n-1}$ satisfies (Property)$_\nu$ and the inequality (\ref{eq:F}) for $\nu< n$. Then, by enlarging $R$ if necessary, it holds that 
\[
\sup_{U_{jk}}\left|\frac{d}{d z_j}(P^{kj}_n-t_{jk}^{-n+1}Q^{kj}_n)\right|\leq \text{the coeff. of}\ X^n\ \text{in}\ \frac{M_3(A(X))^2}{1-6RA(X)}, 
\]
where $M_3:=12R(1+C_0+6RM_1)$. 
\end{lemma}

\begin{proof}
From the equation (\ref{eq:def_P}), we obtain that
\begin{eqnarray}
\left|\frac{d P^{kj}_n}{d z_j}\right|
&\leq& 
\text{the coeff. of}\ X^n\ \text{in}\ \sum_{\nu=2}^{n}\left|\frac{d f^{kj}_\nu(z_j)}{d z_j}\right|\cdot \left(X+\sum_{\lambda=2}^{n-1} |F^j_\lambda| X^\lambda\right)^\nu \nonumber \\
& &\ \ \ \ \ \ \ \ \ \ +\sum_{\nu=2}^{n}\nu\left|f^{kj}_\nu(z_j)\right|\cdot \left(X+\sum_{\lambda=2}^{n-1} |F^j_\lambda| X^\lambda\right)^{\nu-1} \cdot \left(\sum_{\lambda=2}^{n-1} \left|\frac{d F^j_\lambda}{d z_j}\right| X^\lambda\right)\nonumber \\
&\leq& 
\text{the coeff. of}\ X^n\ \text{in}\ \sum_{\nu=2}^\infty \left(1+\frac{\nu}{3}\right)M_1R^\nu\cdot (3A(X))^\nu \nonumber  \\
&\leq& 
\text{the coeff. of}\ X^n\ \text{in}\ \sum_{\nu=2}^\infty 2^\nu M_1R^\nu\cdot (3A(X))^\nu \nonumber  \\
&\leq& 
\text{the coeff. of}\ X^n\ \text{in}\ \frac{M_1R^2(6A(X))^2}{1-6RA(X)} \nonumber 
\end{eqnarray}
(Recall that $V_j\cap C_{\rm sing}=\emptyset$). 
Assuming that $R^{-1}$ is much smaller than the diameters of $V_j$'s, 
we also obtain 
\[
\sup_{U_j\cap U_{k}^*}\left|\frac{\partial F^{kj}_{\nu\lambda}}{\partial z_j}\right| = \sup_{U_j\cap U_{k}^*}\left|\frac{1}{\lambda!}\frac{\partial}{\partial z_j}\frac{\partial^\lambda F^{k}_{\nu}}{\partial w_j^\lambda}\right|
= \sup_{U_j\cap U_{k}^*}\left|\frac{1}{\lambda!}\frac{\partial^\lambda}{\partial w_j^\lambda}\frac{\partial F^{k}_{\nu}}{\partial z_j}\right|
\leq R^\lambda\sup_{V_j\cap V^*_{k}}\left|\frac{\partial F^{k}_\nu}{\partial z_j}\right|, 
\]
which implies the inequality 
\[
\sup_{U_j\cap U_{k}^*}\left|\frac{\partial F^{kj}_{\nu\lambda}}{\partial z_j}\right| 
\leq 2C_0A_\nu R^\lambda. 
\]
Thus we can deduce the inequality
\begin{eqnarray}
\left|\frac{d Q^{kj}_n}{d z_j}\right|
& = &
\text{the coeff. of}\ X^n\ \text{in}\ \sum_{\nu=2}^{\infty} \sum_{\lambda=1}^\infty (2C_0+\lambda)A_\nu R^\lambda\cdot (3A(X))^\lambda\cdot X^\nu \nonumber \\
&\leq&
\text{the coeff. of}\ X^n\ \text{in}\ \sum_{\nu=2}^{\infty} \sum_{\lambda=1}^\infty (1+C_0)2^\lambda A_\nu R^\lambda\cdot (3A(X))^\lambda\cdot X^\nu \nonumber \\
&\leq& 
\text{the coeff. of}\ X^n\ \text{in}\ 6R(1+C_0)\frac{(A(X))^2}{1-6RA(X)} \nonumber 
\end{eqnarray}
on ${U_j\cap U_{k}^*}$ from the equation (\ref{eq:def_Q}). 
Therefore, 
\[
\sup_{U_j\cap U_{k}^*}\left|\frac{d}{d z_j}(P^{kj}_n-t_{jk}^{-n+1}Q^{kj}_n)\right|\leq \text{the coeff. of}\ X^n\ \text{in}\ 6R(6RM_1+1+C_0)\frac{(A(X))^2}{1-6RA(X)}
\]
holds. For each point $p\in U_{jk}$, by using $l$ such that $p\in U_l^*$, we can calculate 
\begin{eqnarray}
\left|\frac{d}{d z_j}(P^{kj}_n-t_{jk}^{-n+1}Q^{kj}_n)(p)\right|&\leq& 
\left|\frac{d}{d z_j}(P^{lj}_n-t_{jl}^{-n+1}Q^{lj}_n)(p)\right|+\left|\frac{d}{d z_j}(P^{lk}_n-t_{kl}^{-n+1}Q^{lk}_n)(p)\right| \nonumber \\
&\leq& 
\text{the coeff. of}\ X^n\ \text{in}\ 12R(6RM_1+1+C_0)\frac{(A(X))^2}{1-6RA(X)}, \nonumber 
\end{eqnarray}
which proves the lemma. 
\end{proof}

Let us fix a $C^\infty$ Hermitian metric $h$ of $K_{\widetilde{C}}$, where $i\colon\widetilde{C}\to C$ is the normalization, 
and consider $\eta_n:=i^*\{(U_{jk}, d(P^{kj}_n-t_{jk}^{-n+1}Q^{kj}_n))\}$. 
The class of $H^1(\widetilde{C}, \mathcal{O}_{\widetilde{C}}(K_{\widetilde{C}}\otimes i^*N_{C/X}^{-n+1}))$ defined by $\eta_n$ is the trivial one. 
Thus, by Lemma \ref{lem:PQ_derivative}, there exists a positive constant $M_4$ depends only on $M_3$ and $h$ such that
\[
\|\eta_n\|_h:=\max_{\nu, \mu}\sup_{\widetilde{U}_{\nu\mu}}|\eta_n^{\nu\mu}|_h\leq \text{the coeff. of}\ X^n\ \text{in}\ \frac{M_4(A(X))^2}{1-6RA(X)}
\]
holds, where $\{\widetilde{U}_\nu\}$ is the open cover of $\widetilde{C}$ as in Remark \ref{rmk:tau_and_a} and 
$\eta_n^{\nu\mu}$ is the $1$-form such that $i^*\{(U_{jk}, d(P^{kj}_n-t_{jk}^{-n+1}Q^{kj}_n))\}=\{(\widetilde{U}_{\nu\mu}, \eta_n^{\nu\mu})\}$ holds. 
Thus we can deduce from \cite[Lemma 2]{KS} that there exists a positive real number $K$ and $\{(\widetilde{U}_\nu, \eta^\nu_n)\}\in \breve{C}^0(\{\widetilde{U}_\nu\}, \mathcal{O}_{\widetilde{C}}(K_{\widetilde{C}}\otimes i^*N_{C/X}^{-n+1}))$ such that 
\[
\eta^\nu_n-t_{jk}^{-n+1}\eta^\mu_n=\eta^{\nu\mu}_n=(i|_{\widetilde{U}_{\nu\mu}})^*d(P^{kj}_n-t_{jk}^{-n+1}Q^{kj}_n),\ \sup_{\widetilde{U}_\nu}|\eta^\nu_n|_h\leq K\|\eta_n\|_h
\]
holds, where $i^{-1}(U_{jk})=\widetilde{U}_{\nu\mu}$. 
Note that $K$ can be taken as a constant which does not depend on $n$ (Here we used the assumption that $i^*N_{C/X}\in \mathcal{E}_0(\widetilde{C})$). 

Now we start constructing functions $G_n^j$, $G_{\pm n}^j$. 
First we consider on $V_j$ such that $U_j\cap C_{\rm sing}=\emptyset$. 
In this case, $U_j$ is isomorphic to $\widetilde{U}_\nu$ for some $\nu$ via $i$. 
We define $G^j_n$ as the extension of the function 
\[
G^j_n(p):=\int_{p_j}^p(i|_{\widetilde{U}_\nu}^{-1})^*\eta^\nu_n
\]
defined on $U_j$. 
Next we consider on $V_k$ such that $U_k\cap C_{\rm sing}\not=\emptyset$. 
In this case, $U_{k}$ is decomposed to two irreducible component $\{x_k=0\}$ and $\{y_k=0\}$. Take $\nu, \mu$ such that $i^{-1}(\{x_k=0\})=\widetilde{U}_\nu$ and $i^{-1}(\{y_k=0\})=\widetilde{U}_\mu$ holds. 
We define the function $G^k_{+n}$ as the extension of the function 
\[
G^k_{+n}(p):=\int_{p_k}^p(i|_{\widetilde{U}_\mu}^{-1})^*\eta^\mu_n
\]
defined on $\{y_k=0\}$, and 
define the function $G^k_{-n}$ as the extension of the function 
\[
G^k_{-n}(p):=\int_{p_k}^p(i|_{\widetilde{U}_\nu}^{-1})^*\eta^\nu_n
\]
defined on $\{x_k=0\}$. 
By constructions, there clearly exists a positive real number $C_1$ depending only on $h$ such that 
$\sup_{U_j}|G^j_n|\leq C_1K\|\eta_n\|_h$ 
($\sup_{U_k}|G^k_{\pm n}|\leq C_1K\|\eta_n\|_h$) 
holds. 
Thus it holds that 
\begin{equation}\label{eq:M_5}
\sup_{U_j}|G^j_n|,\ \sup_{U_k}|G^k_{\pm n}|\leq \text{the coeff. of}\ X^n\ \text{in}\ \frac{M_5(A(X))^2}{1-6RA(X)}, 
\end{equation}
where $M_5:=C_1KM_4$. 

Next, we give the construction of $C^j_n$. Clearly 
\[
C^{kj}_n:=\begin{cases}
    (P^{kj}_n-t_{jk}^{-n+1}Q^{kj}_n)|_{U_{jk}}-(G^j_n-t_{jk}^{-n+1}G^k_n)|_{U_{jk}} & (U_k\cap C_{\rm sing}=\emptyset) \\
    (P^{kj}_n-t_{jk}^{-n+1}Q^{kj}_n)|_{U_{jk}}-(G^j_n-t_{jk}^{-n+1}(G^k_{+n}+G^k_{-n}))|_{U_{jk}} & (U_k\cap C_{\rm sing}\not=\emptyset)
  \end{cases}
\]
is a constant for each $j, k$. 
Thus it follows from Lemma \ref{lem:KS_const} that there exists a positive constant $K_1$ and a $0$-cochain $\{(U_j, C^j_n)\}\in\breve{C}^0(\{U_j\}, \mathbb{C}(N_{C/X}^{-n+1}))$ such that 
\begin{equation}\label{eq:def_C}
C^j_n-t_{jk}^{-n+1}C^k_n=C^{kj}_n,\ 
|C^j_n|\leq K_1\max_{j, k}|C^{kj}_n|
\end{equation}
holds for each $j, k$ (Here we used the assumption that $H^1(C, \mathbb{C}(N_{C/X}^{-n+1}))=0$). 
From the assumption that $N_{C/X}\in\mathcal{E}_0(C)$, it holds that $K_1$ can be taken as a constant which does not depend on $n$. 

Now let us consider
\[
F^j_n:=\begin{cases}
    C^j_n+G^j_n& (U_j\cap C_{\rm sing}=\emptyset) \\
    C^j_n+G^j_{+n}+G^j_{-n} & (U_j\cap C_{\rm sing}\not=\emptyset). 
  \end{cases}
\]
Then, it follows from the argument in the previous subsection that the system $\{F^j_\nu\}_{\nu=2}^n$ satisfies (Property)$_n$. 
As the inequality (\ref{eq:M_5}) and 
\begin{eqnarray}
|C^j_n|&\leq& K_1\max_{j, k}|C^{kj}_n| \nonumber \\
&\leq& K_1\|\{(U_{jk}, P^{kj}_n-t_{jk}^{-n+1}Q^{kj}_n)\}\| + 2K_1\cdot\left(\text{the coeff. of}\ X^n\ \text{in}\ \frac{M_5(A(X))^2}{1-6RA(X)}\right) \nonumber \\
&\leq& \text{the coeff. of}\ X^n\ \text{in}\ K_1(M_2+2M_5)\frac{(A(X))^2}{1-6RA(X)} \nonumber
\end{eqnarray}
holds, letting $K_1>1$, all we have to do is to show the inequality
\[
\text{the coeff. of}\ X^n\ \text{in}\ K_1(M_2+2M_5)\frac{(A(X))^2}{1-6RA(X)}\leq A_n, 
\]
which is clear by letting $M_0:=K_1(M_2+2M_5)$ and $R_0:=6R$. 
\qed

\subsubsection{Convergence of the functional equation (\ref{eq:funceq}) for the case $N_{C/X}\in\mathcal{E}_1(C)$} 

In this subsection, we prove that the formal power series $X+\sum_{n=2}^\infty (\sup_{V_j}|F^j_n|)\cdot X^n$ has a positive radius of convergence for a suitable choice of $\{F^j_n\}$ when $N_{C/X}\in\mathcal{E}_1(C)$. 

Let $M_1$ and $R$ be those in the previous subsection.  
Fixing a sufficiently large positive real number $K_2$, consider the constants 
$M_0:= K_2(M_2+M_5)$ and $R_0:=6R$, where $M_2$ and $M_5$ are the constants appeared in the previous section (Recall that $M_2$ and $M_5$ depend only on the choice of $M_1, R$, and the metric $h$ on $K_{\widetilde{C}}$).
Let $A(X)=X+\sum_{n=2}^\infty A_nX^n$ be the formal power series defined by 
\begin{equation}\label{eq:siegel}
\sum_{n=2}^\infty (M_5+d(\mathcal{O}_C, N_{C/X}^{n-1})^{-1}\cdot M_0)^{-1}\cdot A_nX^n
=\frac{(A(X))^2}{1-R_0A(X)}. 
\end{equation}

\begin{lemma}
$A(X)$ has a positive radius of convergence. 
\end{lemma}

\begin{proof}
Let us denote $\varepsilon_n:=M_5+d(\mathcal{O}_C, N_{C/X}^{n})^{-1}\cdot M_0$. 
It is sufficient to show the following two assertions: 
$(1)$ $-\log \varepsilon_n=O(\log n)$ as $n\to \infty$, 
$(2)$ $\varepsilon_{n-m}^{-1}\leq \varepsilon_n^{-1}+\varepsilon_m^{-1}$ for $m<n$ 
(\cite{Sie}, see also \cite[Lemma 5]{U83}). 
$(1)$ follows from the assumption $N_{C/X}\in\mathcal{E}_1(C)$. 
$(2)$ follows from the inequality 
\begin{eqnarray}
\varepsilon_{n-m}^{-1} 
&=& \frac{d(N_{C/X}^{n}, N_{C/X}^{m})}{M_0+M_5\cdot d(N_{C/X}^{n}, N_{C/X}^{m})}
\leq\frac{d(\mathcal{O}_C, N_{C/X}^{n})+d(\mathcal{O}_C, N_{C/X}^{m})}{M_0+M_5\cdot (d(\mathcal{O}_C, N_{C/X}^{n})+d(\mathcal{O}_C, N_{C/X}^{m}))}. 
\nonumber
\end{eqnarray}
\end{proof}

It is sufficient to show that Claim \ref{claim:F} is also true for $A(X)$ above under the condition $N_{C/X}\in\mathcal{E}_1(C)$.  
As in the previous subsection,  
we will construct $\{F^j_{n}\}$ such that (Property)$_n$ and the inequality (\ref{eq:F}) for $\nu=n$ holds by assuming the existence of $\{F^j_\nu\}_{\nu=2}^{n-1}$ satisfying (Property)$_\nu$ and the inequality (\ref{eq:F}) for $\nu< n$. 
Note that Lemma \ref{lem:PQ}, Lemma \ref{lem:PQ_derivative}, and the inequality (\ref{eq:M_5}) holds also in the present setting. 
Thus we can take $C^j_n$ as (\ref{eq:def_C}). 
However, under the present assumption, we can not take $K_1$ as a constant which depends on $n$ (At least, the existence of such $K_1$ does not followed directly from Lemma \ref{lem:KS_const} in the present setting). 
So here we use Lemma \ref{lem:KS_const} instead of Lemma \ref{lem:KS_const} to deduce that there exists a positive constant $K_2$ which does not depend on $n$ such that 
$|C^j_n|\leq d(C, N_{C/X}^{-n})^{-1}\cdot K_2\cdot\max_{j, k}|C^{jk}_n|$ holds. 
Since
\begin{eqnarray}
|C^j_n| &\leq& \text{the coeff. of}\ X^n\ \text{in}\ d(C, N_{C/X}^{-n})^{-1} K_2(M_2+2M_5)\frac{(A(X))^2}{1-6RA(X)} \nonumber \\
&\leq& \text{the coeff. of}\ X^n\ \text{in}\ (M_5+d(C, N_{C/X}^{-n})^{-1} M_0)\frac{(A(X))^2}{1-6RA(X)}, \nonumber 
\end{eqnarray}
it is sufficient to show the inequality 
\[
\text{the coeff. of}\ X^n\ \text{in}\ (M_5+d(C, N_{C/X}^{-n})^{-1} M_0)\frac{(A(X))^2}{1-6RA(X)} \leq A_n, 
\]
which is clearly followed by the equation (\ref{eq:siegel}). 
\qed

\section{Proof of Theorem \ref{thm:main_finite}}

In this section, we prove Theorem \ref{thm:main_finite}. 
Let $i\colon\widetilde{C}\to C$ be the normalization. 
Fix a neighborhood $V$ of $C$ and a sufficiently fine open covering $\{V_j\}$ of $V$. 
We may assume the open covering $\{U_j\}$ of $C$ defined by $U_j:=V_j\cap V_j$ is enough fine to satisfy the conditions in Remark \ref{rmk:tau_and_a}. 
Let $\{(V_j, w_j)\}$ be a system of order $n$, where $n$ is the type of the pair $(C, X)$. 
From the arguments in \cite[\S 3.4]{U83}, it is sufficient for proving Theorem \ref{thm:main_finite} to construct a real-valued $C^\infty$ function $\varphi_\lambda$ on $V\setminus C$ for each $\lambda>0$ which satisfies the following three properties: 
\begin{description}
  \setlength{\parskip}{0cm} 
  \setlength{\itemsep}{0cm} 
\item[(Property 1)] $\varphi_\lambda(p)=d(p, C)^{-\lambda n}+o(d(p, C)^{-\lambda n})$ as $p\to\infty$. 
\item[(Property 2)] $\varphi_\lambda$ is a strongly psh function on $V\setminus C$ if $\lambda>1$. 
\item[(Property 3)] The complex Hessian of $\varphi_\lambda$ has a positive and negative eigenvalue at each point in $V\setminus C$ if $\lambda<1$. 
\end{description}

Note that, by Proposition \ref{prop:hodge}, the natural map $H^1(C, \mathbb{C}(N_{C/X}^m))\to H^1(C, \mathcal{O}_C(N_{C/X}^m))\oplus H^1(C, \overline{\mathcal{O}}_C(N_{C/X}^m))$ is isomorphism for all integer $m$ (Here we used the assumption that $G(C)$ is a tree and that $N_{C/X}=\mathcal{O}_C$). 
Thus we can run the same argument as in \cite[\S 3.3]{U83} and then construct a function $s\colon V\setminus C\to\mathbb{R}$ such that 
\begin{eqnarray}
&&s|_{V_j}=|w|^{-2n}-w^{-n}g-\overline{w}^{-n}\overline{g}-w^{-n}\cdot\overline{\left(\sum_{a, b\geq 0,\ 1\leq a+b\leq n}\varphi_{ab}w^a\overline{w}^b\right)}\nonumber \\
&&\hskip30mm -\overline{w}^{-n}\cdot\left(\sum_{a, b\geq 0,\ 1\leq a+b\leq n}\varphi_{ab}w^a\overline{w}^b\right)
+|g|^2+|w|\cdot\alpha\nonumber
\end{eqnarray}
holds for some holomorphic function $g=g_j$ defined on $V_j$, 
pluriharmonic functions $\varphi_{ab}=\varphi_{j|ab}$ defined on $V_j$, and some $C^\infty$ function $\alpha=\alpha_j$ defined on $V_j$, where we are denoting by $w$ the function $w_j$. 
Moreover, it can also be assumed that the element of $H^1(C, \mathcal{O}_C(N_{C/X}^{-n}))$ defined by 
$\delta\{(U_j, \overline{g_j}|_{U_j})\}$ coincides with the class defined by $\{(U_{jk}, (w_j^{-n}-t_{jk}^{-n}w_k^{-n})|_{U_{jk}})\}$, 
where $t_{jk}\in U(1)$ is a transition function of $N_{C/X}$ on $U_{jk}$. 
From this fact and the assumption that $u_n(C, X)|_{{C}_\nu}\not=0\in H^1(C_\nu, \mathcal{O}_C|_{{C}_\nu})$ for each irreducible component $C_\nu$ of $C$, it follows that $dg_j|_{U_j}\not\equiv 0$ holds for all $U_j$. 
Thus it follows from the calculation in \cite[\S 3.3]{U83} that, after a small modification, $s^{\frac{\lambda}{2}}$ satisfies 
(Property 1), (Property 2), and (Property 3) on each $V_j$ such that $V_j\cap C_{\rm sing}=\emptyset$. 

Thus all we have to do is to modify the function $s^{\frac{\lambda}{2}}$ around each singular point $p\in C_{\rm sing}$ and show it satisfies (Property 1), (Property 2), and (Property 3). 
We assume $C_{\rm sing}=\{p\}\subset U_k$ for simplicity. 
In the rest of this section, we always consider on $V_k$ and omit the index $k$ ($w=w_k$, $g=g_k$ for example). 
Fix a coordinates system $(x, y)$ such that $w=x\cdot y$. 
Note that we may assume $g$ can be written in the form $g(x, y)=g_1(x)+g_2(y)$, where $g_1$ ($g_2$) is a holomorphic function depending only on the valuable $x$ ($y$, respectively). 
Moreover, we may assume that $\{dg_1=0\}\subset \{x=0\}$ and $\{dg_2=0\}\subset \{y=0\}$ hold by shrinking $U_k$ if necessary. 
By using a coordinates system $(w, z):=(x\cdot y, y)=(w, y)$ on $\{x\cdot y\not=0\}$, we can calculate the complex Hessian
\[
H_\lambda:=\left[
    \begin{array}{cc}
      (s^{\frac{\lambda}{2}})_{w\overline{w}} & (s^{\frac{\lambda}{2}})_{w\overline{z}} \\
      (s^{\frac{\lambda}{2}})_{z\overline{w}} & (s^{\frac{\lambda}{2}})_{z\overline{z}}
    \end{array}
  \right]
\]
of $s^{\frac{\lambda}{2}}$ as follows: 
\[
H_\lambda=\frac{\lambda}{2} |w|^{-(\lambda-2)n}\cdot
\left[
    \begin{array}{cc}
      \frac{\lambda}{2}n^2|w|^{-2n-2}\cdot (1+O(|w|)) & \left(\frac{\lambda}{2}-1\right)nw^{-1}\overline{w}^{-n}\overline{g_z}\cdot (1+O(|w|)) \\
     \left(\frac{\lambda}{2}-1\right)n\overline{w}^{-1}{w}^{-n}{g_z}\cdot (1+O(|w|)) & \frac{\lambda}{2}|g_z|^2
    \end{array}
  \right]. 
\]
Thus it holds that 
\[
{\rm det}\,H_\lambda
=\frac{(\lambda-1)\lambda^2n^2}{4} |g_z|^2|w|^{-2(\lambda n-n+1)}(1+O(|w|)). 
\]
Note that, as 
\[
g_z=(g_1+g_2)_z=(g_1)_x\cdot\left(\frac{w}{z}\right)_z+(g_2)_y=\frac{-x\cdot (g_1)_x+y\cdot (g_2)_y}{y}
\]
does not vanishes on each point in $\{x\not=0\}\cup\{y\not=0\}$, 
it is sufficient to modify $s^{\frac{\lambda}{2}}$ only around $p$. 
Letting $\rho(x, y)$ be a $C^\infty$ cut-off function such that $\rho|_{\{|x|^2+|y|^2<\delta\}}\equiv 1$ and ${\rm Supp}\,\rho\subset \{|x|^2+|y|^2<2\delta\}$ hold for some sufficiently small number $\delta>0$, 
consider the function $\varphi_\lambda$ defined by 
\[
\varphi_\lambda:= \begin{cases}
    s^{\frac{\lambda}{2}}+\varepsilon\rho(x, y)\cdot\left(|z|^2+\frac{|w|^2}{|z|^2}\right)\cdot|w|^{-(\lambda-2)n} & (\lambda>1) \\
    s^{\frac{\lambda}{2}}-\varepsilon\rho(x, y)\cdot\left(|z|^2+\frac{|w|^2}{|z|^2}\right)\cdot|w|^{-(\lambda-2)n} & (0<\lambda<1),
  \end{cases}
\]
where $\varepsilon$ is a sufficiently small positive real number. 
As the complex Hessian of the function $R_\lambda(w, z):=\left(|z|^2+\frac{|w|^2}{|z|^2}\right)\cdot|w|^{-(\lambda-2)n}$ can be calculated as 
\[
\left[
    \begin{array}{cc}
      (R_\lambda)_{w\overline{w}} & (R_\lambda)_{w\overline{z}} \\
      (R_\lambda)_{z\overline{w}} & (R_\lambda)_{z\overline{z}}
    \end{array}
  \right]
=|w|^{-(\lambda-2)n}
\left[
    \begin{array}{cc}
     O\left(\frac{1}{|w|^2}\right) & o\left(\frac{1}{w^2}\right) \\
      o\left(\frac{1}{\overline{w}^2}\right)& 1+\frac{|w|^2}{|z|^4}
    \end{array}
  \right], 
\]
we can conclude that, by shrinking $\varepsilon$ if necessary, the above function $\varphi_\lambda$ satisfies (Property 1), (Property 2), and (Property 3). 
\qed

\section{Proof of Theorem \ref{thm:main_2}}

In this section, we prove Theorem \ref{thm:main_2}. 
Let $C=C_1\cup C_2\cup\dots\cup C_N$ be the irreducible decomposition of $C$ and 
let $C_N\cap C_1=\{p_1\}$, $C_{\nu}\cap C_{\nu+1}=\{p_\nu\}$ ($\nu=2, 3, \dots, N$). 
In the following, we sometimes denote $p_N$ by $p_{0}$ and $C_N$ by $C_{0}$. 
Fix a neighborhood $V$ of $C$ and 
a sufficiently fine open covering $\{V_j\}$ of $V$. 
Denote by $\{U_j\}$ the induced open covering of $C$: $U_j:=V_j\cap V_j$. 
We may assume that $\{U_j\}$ satisfies the conditions in Remark \ref{rmk:tau_and_a}. 
Denote by $U_{k_\nu}$ the open set which includes $p_\nu$ as an element. 
Each $U_{k_\nu}$ has two irreducible components 
$U_{k_\nu}^{(\nu-1)}\subset C_{\nu-1}$ and $U_{k_\nu}^{(\nu)}\subset C_{\nu}$. 
Denoting by $U_j^{(\nu)}$ the open set $U_j$ by using $\nu$ which satisfies $U_j\in C_\nu$ for each $j\not\in\{k_1, k_2, \dots, k_N\}$, 
$\{U_j^{(\nu)}\mid j=k_{\nu}\ \text{or}\ j=k_{\nu+1}\ \text{or}\ U_j\subset C_\nu\}$ defines an open covering of $C_\nu$ for each $\nu=0, 1, \cdots, N-1$. 

From the arguments in \cite[\S 3.4]{U83}, it is sufficient for proving Theorem \ref{thm:main_2} to construct a real-valued $C^\infty$ function $\varphi_\lambda$ on $V\setminus C$ for each $\lambda>0$ which satisfies the following three properties: 
\begin{description}
  \setlength{\parskip}{0cm} 
  \setlength{\itemsep}{0cm} 
\item[(Property 1)] $\varphi_\lambda(p)=(\log d(p, C))^{2_\lambda}+o((\log d(p, C))^{2_\lambda})$ as $p\to\infty$. 
\item[(Property 2)] $\varphi_\lambda$ is a strongly psh function on $V\setminus C$ if $\lambda>1$. 
\item[(Property 3)] The complex Hessian of $\varphi_\lambda$ has a positive and negative eigenvalue at each point in $V\setminus C$ if $\lambda<1$. 
\end{description}

\begin{lemma}\label{lem:cycle}
For a suitable positive number $\alpha\in\mathbb{R}$ and $\{t_{jk}\}\subset \mathbb{C}^*$ such that $N_{C/X}=[\{(U_{jk}, t_{jk})\}]$, 
the following conditions hold: 
$|t_{jk}|=1$ holds if $j, k\not=k_1$, or if $j\not=k_1$ and $U_j\subset C_1$ hold. 
$|t_{jk_1}|=\alpha$ holds if $U_j\subset C_N$. 
\end{lemma}

\begin{proof}
It can be shown by the same argument as in the proof of Lemma \ref{lem:topologically_trivial_lb} by using 
the normalization $i_1\colon\widetilde{C}_1\to C$ of$C$ at $p_1$ instead of the normalization
$i\colon\widetilde{C}\to C$ 
(Note that, by Lemma \ref{lem:tree}, $i_1^*N_{C/X}$ is flat). 
\end{proof}

Fix $\alpha, \{t_{jk}\}$ as in Lemma \ref{lem:cycle}. 
Note that, by the assumption that $N_{C/X}$ is not flat, $\alpha\not=1$ holds. 
Fix also a system $\{(V_j, w_j)\}$ of order $4$. 

\begin{lemma}\label{lem:system_fxy}
There exists a nowhere vanishing holomorphic function $f_j$ on $V_j$ for each $j\not\in\{k_1, k_2, \dots, k_N\}$ and 
a coordinates system $(x_\nu, y_\nu)$ of $V_{k_\nu}$ for each $k_\nu (\nu=1, 2, \dots, N)$ 
such that the following conditions hold: \\
$(i)$ For each $j\not\in\{k_1, k_2, \dots, k_N\}$, $f_j$ depends only on the valuable $z_j$, where $z_j$ is a function such that $(w_j, z_j)$ is a coordinates system of $V_j$. \\
$(ii)$ For each $j, k$ such that $U_{j}, U_k\subset C_\nu$, there exists an element $s_{jk}\in U(1)$ such that $s_{jk}f_k=f_j+O(w_j)$ holds on $V_{jk}$. \\
$(iii)$ $\{x_\nu=0\}=U_{k_\nu}^{(\nu-1)}$ and $\{y_\nu=0\}=U_{k_\nu}^{(\nu)}$ hold on $V_{k_\nu} (\nu=1, 2, \dots, N)$. \\
$(iv)$ For each $j$ and $\nu=1, 2, \dots, N$ such that $U_{jk_\nu}\not=\emptyset$, there exists an element $s_{jk_\nu}\in U(1)$ such that $s_{jk_\nu}y_\nu^{-1}=f_j+O(w_j)$ holds if $U_j\subset C_{\nu-1}$, and that $s_{jk_\nu}x_\nu=f_j+O(w_j)$ holds if $U_j\subset C_{\nu}$. 
\end{lemma}

\begin{proof}
As $C_\nu$ is non-singular, 
$\mathcal{O}_{C_\nu}(p_\nu-p_{\nu+1})$ is flat for each $\nu=0, 1, \dots, N-1$. 
Thus there exists $s_{jk}\in U(1)$ for each $j, k$ and a holomorphic function $F_j^{(\nu)}$ on $U_j^{(\nu)}$ for each $j$ such that $F_j^{(\nu)}=s_{jk}F_k^{(\nu)}$ holds on $U_{jk}^{(\nu)}$. 
$f_j, x_\nu, y_\nu^{-1}$ can be constructed by extending each $F_j^{(\nu)}$ to $V_j$ in a suitable manner. 
\end{proof}

\begin{remark}\label{rmk:S_const}
In Lemma \ref{lem:system_fxy}, 
we may additionally assume that $S_\nu|_{U_{k_\nu}}$ is a constant function for each $\nu$, 
where $S_\nu:=(w_{k_\nu}/(x_\nu\cdot y_\nu))$. 
Here we will show this assertion by modifying $(x_\nu, y_\nu)$ as in Lemma \ref{lem:system_fxy} and constructing a new coordinates system $(\widetilde{x}_\nu, \widetilde{y}_\nu)$ with $(w_{k_\nu}/(\widetilde{x}_\nu\cdot \widetilde{y}_\nu))|_{U_{k_\nu}}\equiv q_\nu$, 
where $q_\nu:=S_\nu(p_\nu)$. 
Let $Q_\nu$ be a holomorphic function defined on $V_{k_\nu}$ which coincides with $S_\nu/q_\nu$ on $U_{k_\nu}^{(\nu)}$ and with the constant function $1$ on $U_{k_\nu}^{(\nu-1)}$, and consider $\widetilde{y}_\nu:=Q_\nu\cdot y_\nu$. 
$\widetilde{y}_\nu$ coincides with $y_\nu$ on ${U_{k_\nu}^{(\nu-1)}}$ and satisfies that $\widetilde{S}_\nu|_{U_{k_\nu}^{(\nu)}}\equiv q_\nu$, where 
$\widetilde{S}_\nu:=(w_{k_\nu}/(x_\nu\cdot \widetilde{y}_\nu))$. 
Let $\widetilde{Q}_\nu$ be a holomorphic function defined on $V_{k_\nu}$ which coincides with $\widetilde{S}_\nu/q_\nu$ on 
$U_{k_\nu}^{(\nu-1)}$ and with the constant function $1$ on $U_{k_\nu}^{(\nu)}$, 
and consider $\widetilde{x}_\nu:=\widetilde{Q}_\nu\cdot x_\nu$. 
$\widetilde{x}_\nu$ coincides with $x_\nu$ on ${U_{k_\nu}^{(\nu)}}$ and satisfies that 
$(w_{k_\nu}/(\widetilde{x}_\nu\cdot \widetilde{y}_\nu))|_{U_{k_\nu}}\equiv q_\nu$ holds, 
which shows the assertion. 
\end{remark}

\begin{lemma}\label{lem:q_const}
There exists a constant $q\in\mathbb{C}^*$ and functions $\{f_j\}, \{(x_\nu, y_\nu)\}$ as in 
Lemma \ref{lem:system_fxy} such that 
$S_\nu|_{U_{k_\nu}}\equiv q$ holds for each $\nu\in\{1, 2, \dots, N\}$, where $\widetilde{S_\nu}:=w_{k_\nu}/(x_\nu\cdot y_\nu)$. 
\end{lemma}

\begin{proof}
By Remark \ref{rmk:S_const}, we may assume that the function $S_\nu=w_{k_\nu}/(x_\nu\cdot y_\nu)$ coincides with the constant function $q_\nu$ on $U_\nu$. 
Fix a complex number $q$ such that $q^N=q_1\cdot q_2\cdot \cdots q_N$ 
and consider a sequence $\{a_1, a_2, \dots, a_N\}\subset\mathbb{C}^*$ which is defined by 
$a_1:=1$ and $a_\nu:=q^{-1}\cdot q_\nu\cdot a_{\nu-1}$. 
Let us denote by $\widetilde{f}_j$ the function $a_\nu f_j$ for each $j$ such that $U_j\subset C_\nu$, 
and by $(\widetilde{x}_\nu, \widetilde{y}_\nu)$ the new coordinates system of $V_{k_\nu}$ defined by 
$\widetilde{x}_\nu:=a_\nu\cdot x_\nu, \widetilde{y}_\nu:=a_{\nu-1}^{-1}\cdot y_\nu$. 
Then the assertion follows from the fact that 
$(w_{k_\nu}/(\widetilde{x}_\nu\cdot \widetilde{y}_\nu))|_{U_\nu}$ is the constant function $q_\nu\cdot a_{\nu-1}\cdot a_\nu^{-1}$. 
\end{proof}

Now we start constructing $\varphi_\lambda$ by using $\{f_j\}, \{(x_\nu, y_\nu)\}, q$ as in Lemma \ref{lem:q_const}. 
By replacing each $w_j$ with $q^{-1}\cdot w_j$, we assume that $q=1$. 
Define the function $\varphi_j$ on each $V_j\subset U_j$ as follows: 
\[
\varphi_j:=(\log |w_j|)^2+\frac{N-2\nu}{N}\cdot\log\alpha\cdot\log |w_j|+\frac{2}{N}\cdot\log\alpha\cdot\log |f_j|
\]
if $U_j\subset C_\nu$ holds for some $\nu\in\{1, 2, \dots, N\}$, and 
\[
\varphi_{k_\nu}:=(\log |w_{k_\nu}|)^2+\frac{N-2(\nu-1)}{N}\cdot\log\alpha\cdot\log |w_{k_\nu}|-\frac{2}{N}\cdot\log\alpha\cdot\log |y_\nu|
\]
for each $\nu=1, 2, \dots, N$. 
Let $\{\rho_j\}$ be a system of $C^\infty$ functions $\rho_j\colon V_j\to[0, 1]\subset\mathbb{R}$ such that 
${\rm Supp}\,\rho_j\subset V_j$ holds for each $j$ and $(\sum_j\rho_j)|_{V_0}\equiv 1$ holds for some neighborhood $V_0$ of $C$ in $V$. 
Since the calculation as in \cite[p. 687]{U91} shows that 
\[
\eta_{jk}(p):=\begin{cases}
    \varphi_k(p)-\varphi_j(p) & (p\in V_{jk}\setminus U_{jk}) \\
    0 & (p\in U_{jk})
  \end{cases}
\]
is a $C^2$ class function defined on $V_{jk}$ for each $j, k$ such that $U_{jk}\not=\emptyset$. 
Let us denote by $\varphi$ the function on $V\setminus C$ such that $\varphi|_{V_j}=\varphi_j+\sum_k\rho_k\cdot \eta_{jk}$ for each $j$. 

We will construct a function $\varphi_\lambda$ with (Property 1), (Property 2), and (Property 3) by modifying the function $\varphi^\lambda$. 
First we obtain from the same calculation as in \cite[p. 689]{U91} that
\[
\frac{\varphi^{-\lambda+2}}{\lambda}\left[
    \begin{array}{rr}
      (\varphi^\lambda)_{w\overline{w}} & (\varphi^\lambda)_{w\overline{z}} \\
      (\varphi^\lambda)_{z\overline{w}} & (\varphi^\lambda)_{z\overline{z}}
    \end{array}
  \right]
= \left[
    \begin{array}{cc}
      \left(\lambda-\frac{1}{2}\right)\frac{(\log|w|)^2}{|w|^2}(1+o(1)) & \left(\lambda-1\right)\frac{-B\log|w|}{w\overline{z}}(1+o(1)) \\
      \left(\lambda-1\right)\frac{-B\log|w|}{|\overline{w}z|^2}(1+o(1)) & \left(\lambda-1\right)\frac{B^2}{|z|^2}(1+o(1))
    \end{array}
  \right]
\]
holds on each $V_{k_\nu}\setminus U_{k_\nu}$ for each $\nu=1, 2, \dots, N$, where $(w, z)=(w_{k_\nu}, y_\nu)$ and $B=\frac{\log\alpha}{N}$. Note the determinant of the above matrix is
\[
\frac{\lambda-1}{2}\frac{B^2(\log|w|^2)}{|wz|^2}
=\frac{\lambda-1}{2}\cdot\frac{(\log\alpha)^2}{N^2}\cdot\frac{(\log|w|^2)}{|wz|^2}. 
\]
For each $j$ such that $U_j\subset C_\nu$, we obtain from the same calculation as in \cite[p. 688]{U91} that the above matrix with $(w, z):=(w_j, z_j)$ is
\[
\left[
    \begin{array}{cc}
      \left(\lambda-\frac{1}{2}\right)\frac{(\log|w|)^2}{|w|^2}(1+o(1)) & \left(\lambda-1\right)\frac{\log\alpha}{N}\cdot\frac{\log|w|}{w}\overline{\left(\frac{(f_j)_z}{f_j}\right)}(1+o(1)) \\
      \left(\lambda-1\right)\frac{\log\alpha}{N}\cdot\overline{\left(\frac{\log|w|}{w}\right)}\frac{(f_j)_z}{f_j}(1+o(1)) & \left(\lambda-1\right)\frac{(\log\alpha)^2}{N^2}\cdot\left|\frac{(f_j)_z}{f_j}\right|^2
    \end{array}
  \right], 
\]
whose determinant is 
\[
\frac{\lambda-1}{2}\frac{(\log\alpha)^2}{N^2}\cdot\left|\frac{(f_j)_z}{f_j}\right|^2\frac{(\log|w|)^2}{|w|^2}(1+o(1)). 
\]
Thus it turns out that the function $\varphi_\lambda:=\varphi^\lambda$ satisfies (Property 1), (Property 2), and (Property 3) if each function $(f_j)_{z_j}$ has no zero. 

In the rest of this section, we construct a function $\varphi_\lambda$ with (Property 1), (Property 2), and (Property 3) when there uniquely exists a open set $U_j$ ($j\not\in\{k_1, k_2, \dots, k_N\}$) such that $(f_j)_{z_j}$ vanishes only on $\{z_j=0\}$, for simplicity. 
Letting $\chi(z_j)$ be a $C^\infty$ cut-off function such that $\chi|_{\{|z_j|<\delta\}}\equiv 1$ and ${\rm Supp}\,\chi\subset \{|z_j|^2<2\delta\}$ hold for some sufficiently small number $\delta>0$, 
consider the function $\varphi_\lambda$ defined by 
\[
\varphi_\lambda:= \begin{cases}
    \varphi(w_j, z_j)^\lambda+\varepsilon\lambda\cdot\chi(z_j)\cdot(-\log|w_j|)^{2(\lambda-2)}\cdot|z_j|^2 & (\lambda>1) \\
    \varphi(w_j, z_j)^\lambda-\varepsilon\lambda\cdot\chi(z_j)\cdot(-\log|w_j|)^{2(\lambda-2)}\cdot|z_j|^2 & (0<\lambda<1), 
  \end{cases}
\]
where $\varepsilon>0$ is a sufficiently small positive real number. 
As 
\[
{(-\log|w|)^{-\lambda+2}}\left[
    \begin{array}{rr}
      (R_\lambda)_{w\overline{w}} & (R_\lambda)_{w\overline{z}} \\
      (R_\lambda)_{z\overline{w}} & (R_\lambda)_{z\overline{z}}
    \end{array}
  \right]
=  \left[
    \begin{array}{cc}
      O\left(\frac{1}{|w|^2(-\log|w|)^2}\right) & O\left(\frac{1}{w(-\log|w|)}\right) \\
     O\left(\frac{1}{\overline{w}(-\log|w|)}\right) & \left(\chi(z)\cdot|z|^2\right)_{z\overline{z}}
    \end{array}
  \right]
\]
holds for $R_\lambda:=(-\log|w_j|)^{2(\lambda-2)}\cdot|z_j|^2$ and $(w, z):=(w_j, z_j)$, 
we can conclude that, by shrinking $\varepsilon$ if necessary, the above function $\varphi_\lambda$ satisfies (Property 1), (Property 2), and (Property 3). 
\qed

\section{Proof of Theorem \ref{thm:maincor} and application to the blow-up of the projective plane at nine points}

\subsection{Minimal singular metrics of the anti-canonical bundle of the blow-up of the projective plane at nine points}

In this section, we will prove Theorem \ref{thm:maincor}. 
We also study singular Hermitian metrics with semi-positive curvature on $K_X^{-1}$ of the blow-up $X$ of $\mathbb{P}^2$ at nine points in arbitrary position by applying Theorem \ref{thm:main} and Theorem \ref{thm:main_2}, 
and prove the following: 

\begin{theorem}\label{thm:9ptbup}
Let $p_1, p_2, \dots, p_9\in\mathbb{P}^2$ be $9$ points different from each other, $\pi\colon X\to\mathbb{P}^2$ be the blow-up at $\{p_j\}_{j=1}^9$. Then one of the following five assertions holds: \\
$(i)$ $K_X^{-1}$ is semi-ample (i.e. $K_X^{-n}$ is generated by global sections for some integer $n$). \\
$(ii)$ $K_X^{-1}$ is not semi-ample, however it is semi-positive (i.e. $K_X^{-1}$ admits a $C^\infty$ Hermitian metric with semi-positive curvature). \\
$(iii)$ $K_X^{-1}$ is nef and there exists a section $f\in H^0(X, K_X^{-1})\setminus\{0\}$ such that the singular Hermitian metric $|f|^{-2}$ is a minimal singular metric (i.e.\hskip1mma metric with the mildest singularities among singular Hermitian metrics of $K_{X}^{-1}$ whose local weights are psh). In this case, $K_X^{-1}$ is not semi-positive. \\
$(iv)$ $K_X^{-1}$ is nef, and there exists a compact curve $C\subset X$ with nodes such that $N_{C/X}\in\mathcal{P}_0(C)\setminus(\mathcal{E}_0(C)\cup\mathcal{E}_1(C))$ and $K_X^{-1}=\mathcal{O}_X(C)$. \\
$(v)$ $K_X^{-1}$ is not nef, and the nef part of the Zariski decomposition of $K_X^{-1}$ is semi-ample. 
\end{theorem}

For the precise definition of minimal singular metrics, see \cite[Definition 1.4]{DPS00}. 

\begin{remark}
It also follows from the proof of Theorem \ref{thm:9ptbup} blow that none of the five conditions $(i), (ii), \dots, (v)$ in Theorem \ref{thm:9ptbup} can be removed, 
since there exists configurations for which each condition is realized. 
\end{remark}

\begin{remark}\label{rmk:msm}
Note that, except the case of Theorem \ref{thm:9ptbup} $(iv)$, we can determine the concrete expression of a minimal singular metric on $K_X^{-1}$. 
For the case $(i)$ and $(ii)$, $K_X^{-1}$ is semi-positive and thus a minimal singular metric can be taken as a $C^\infty$ one and thus it has no singularity. 
For the case $(v)$, the metric
$h_{\rm min}=h_P\otimes\prod_j|g_j|^{-2a_j}$ is a minimal singular metric of $K_X^{-1}$, 
where $h_P$ is a $C^\infty$ Hermitian metric on the nef part $P$ of Zariski decomposition of $K_X^{-1}$ with semi-positive curvature, 
$N=\sum_ja_jD_j$ is the negative part, 
and $g_j\in H^0(X, \mathcal{O}_X(D_j))$ is the canonical section for each $j$ 
(The minimal singularity of the above metric directly follows from the fact that, for all closed (semi-)positive current $T$ in the class $c_1(K_X^{-1})$, the Lelong number $\nu(T, D_j)$ of 
$T$ along $D_j$ is greater than or equal to $a_j$ for each $j$ (see \cite[p. 54]{Bo}). 
\end{remark}

\subsection{Preliminary for the proof of Theorem \ref{thm:maincor} and Theorem \ref{thm:9ptbup}}

In this subsection, we give some lemmata and propositions needed in the proof of Theorem \ref{thm:maincor} and Theorem \ref{thm:9ptbup}. 
First we show the following: 

\begin{lemma}\label{lem:torsion_9ptbup}
Let $X$ be a non-singular rational surface and $C\subset X$ be a reduced compact curve with only nodes such that $N_{C/X}\in\mathcal{E}_0(C)$. 
Assume one of the following conditions: \\
$(1)$ $C$ is a non-singular elliptic curve. \\
$(2)$ Each component of the normalization of $C$ is a rational curve, and 
the Euler number of the dual graph $G(C)$ is equal to $0$. \\
Then $\mathcal{O}_X(C)$ is semi-ample. 
\end{lemma}

\begin{proof}
Let $n\geq 1$ be the minimum integer such that $N_{C/X}^n=\mathcal{O}_C$. 
By considering the exact sequence 
\[
H^0(X, \mathcal{O}_X(nC))\to H^0(C, N_{C/X}^n)\to H^1(X, \mathcal{O}_X((n-1)C))
\]
induced from $0\to\mathcal{O}_X((n-1)C)\to\mathcal{O}_X(nC)\to\mathcal{O}_X(nC)\otimes\mathcal{O}_X/\mathcal{O}_X(-C)\to 0$, it is sufficient to show that $H^1(X, \mathcal{O}_X((n-1)C))=0$ holds. 
Note that $H^1(C, N_{C/X}^r)=0$ holds for each $0\leq r\leq n-1$ 
(It is clear for the case $(1)$. For the case $(2)$, we can prove it by using the same calculation as in the proof of Proposition \ref{prop:hodge}). 
Thus we obtain $H^1(X, \mathcal{O}_X((n-1)C))\cong H^1(X, \mathcal{O}_X)$ from the argument in \cite[p. 38]{N}, which shows the lemma. 
\end{proof}

In the proof of Theorem \ref{thm:maincor} and Theorem \ref{thm:9ptbup}, we use Theorem \ref{thm:main} in the following form: 

\begin{corollary}\label{cor:main_P1}
Let $X$ be a non-singular surface, $C\subset X$ be a reduced compact curve with only nodes such that
each component of the normalization of $C$ is a rational curve. 
Assume that the dual graph $G(C)$ is a cycle graph and that $N_{C/X}\in\mathcal{E}_1(C)$ holds. 
Then $\mathcal{O}_V(C)$ is a flat line bundle for some neighborhood $V$ of $C$ in $X$. 
\end{corollary}

\begin{proof}
We obtain $H^1(C, \mathbb{C}(N_{C/X}^{n}))=0$ and $H^1(C, N_{C/X}^{n})=0$ hold for each $n$ 
from Proposition \ref{prop:hodge} and the calculation as in the proof of it. 
Note that especially it holds that $u_n(C, X)=0$ for each $n$. 
Thus we can apply Theorem \ref{thm:main}, which shows the corollary. 
\end{proof}

In the proof of Theorem \ref{thm:9ptbup}, it is also needed to treat a divisor $D$ of a non-singular surface $X$ which can be written in the form
\[
D=aC+a_1E_1+a_2E_2+\dots +a_NE_N\ \ \ (a, a_1, a_2, \dots, a_N>0), 
\]
where $C, E_1, E_2, \dots, E_N$ are non-singular compact curves embedded in $X$ which satisfy the following conditions: 
\begin{itemize}
  \setlength{\parskip}{0cm} 
  \setlength{\itemsep}{0cm} 
\item There exists $N$ points $p_1, p_2, \dots, p_N\in C$ different from each other such that, for each $\nu=1, 2, \dots, N$, $C$ intersects $E_\nu$ at $p_\nu$ transversally. 
\item $E_\nu\cap E_\mu=\emptyset$ holds for each $\nu\not=\mu\in\{1, 2, \dots, N\}$. 
\item The line bundle $N_{D/S}:=\mathcal{O}_S(D)|_{{\rm Supp}\,D}$ is topologically trivial, where ${\rm Supp}\,D=C\cup E_1\cup E_2\cup\dots\cup E_N$. 
\item The greatest common divisor of $(a, a_1, a_2, \dots, a_N)$ is equal to $1$. 
\end{itemize}
Denote $b$ the integer $-(C^2)$ and by $b_\nu$ the integer $-(E_\nu^2)$ for each $\nu=1, 2, \dots, N$. 
Then it is clear from $(D^2)=0$ that $b=(a_1+a_2+\dots +a_N)/a$ and $b_\nu=a/{a_\nu}$ hold ($\nu=1, 2, \dots, N$). 
Fix a connected open neighborhood $W$ of ${\rm Supp}\,D$ which is a deformation retract of ${\rm Supp}\,D$. 

\begin{lemma}\label{lem:existence_L}
Let $(X, D), W$ be as above. 
Then there exists a holomorphic line bundle $L$ on $W$ such that $L^{a}=\mathcal{O}_W(\sum_{\nu=1}^Na_\nu E_\nu)$ holds. 
\end{lemma}

\begin{proof}
Denote by $L'$ the line bundle $\mathcal{O}_W(-C)$. 
By considering the isomorphism $H^2({\rm Supp}\,D, \mathbb{Z})\cong H^2(W, \mathbb{Z})$, we obtain that 
$\mathcal{O}_W(D)=(L')^{a}\otimes \mathcal{O}_W(\sum_{\nu=1}^Na_\nu E_\nu)$ is topologically trivial. 
Thus there exists a topologically trivial line bundle $N'$ on $W$ such that 
$(N')^{a}=(L')^{a}\otimes \mathcal{O}_W(\sum_{\nu=1}^Na_\nu E_\nu)$. 
We can prove the lemma by letting $L:=L'\otimes N'$. 
\end{proof}

In the following, we fix a line bundle $L$ as in Lemma \ref{lem:existence_L}. 

\begin{lemma}\label{lem:existence_Wtild}
Let $(X, D), W, L$ be as above. 
Then there exists a connected non-singular complex surface $\widetilde{W}$ and 
a covering map $p\colon \widetilde{W}\to W$ with degree $a$ which satisfies the following conditions: \\
$(i)$ $p^{-1}(E_\nu)$ is the union of non-singular connected compact curves $\{\widetilde{E}_\nu^{(\lambda)}\}_{\lambda=1}^{a_\nu}$ for each $\nu$. 
\\
$(ii)$ $p$ is ramified along $\widetilde{E}_\nu^{(\lambda)}$ with ramification index $b_\nu$ for each $\nu$ and $\lambda$, and is unramified outside of them. \\
$(iii)$ $\mathcal{O}_W(\sum_{\nu=1}^N\sum_{\lambda=1}^{a_\nu}\widetilde{E}_\nu^{(\lambda)})=p^*L$ holds. \\
$(iv)$ $\mathcal{O}_{\widetilde{W}}(\widetilde{D})=p^*N'$ and $\mathcal{O}_{\widetilde{W}}(a\widetilde{D})=p^*\mathcal{O}_W(D)$ hold, where $N'$ is the line bundle appeared in the proof of Lemma \ref{lem:existence_L} and 
\[
\widetilde{D}:=p^{-1}(C)+\sum_{\nu=1}^N\sum_{\lambda=1}^{a_\nu}\widetilde{E}_\nu^{(\lambda)}. 
\] 
\end{lemma}

\begin{proof}
$p\colon \widetilde{W}\to W$ can be constructed as the normalization of the cyclic cover of $W$ defined by the line bundle $L$. 
\end{proof}

For the above $(X, D), W$, we can show the following two propositions by applying \cite[Theorem 1, 2, 3]{U83} to the pair $(\widetilde{W}, \widetilde{D})$ as in Lemma \ref{lem:existence_Wtild} 
(We here remark that the idea to apply Ueda theory on a finite cover is pointed out by Prof. Tetsuo Ueda). 

\begin{proposition}\label{thm:SNCuedath_1}
Let $(X, D), W, L$ be as above. 
Assume that each $E_\nu$ is a rational curve, 
$N_{D/W}=\mathcal{O}_{{\rm Supp}\,D}$ holds, and that the pair $(\widetilde{W}, \widetilde{D})$ is of infinite type. 
Then, after shrinking $W$ if necessary, $\mathcal{O}_W(D)$ is a flat line bundle. 
\end{proposition}

\begin{proof}
Note that each $\widetilde{E}^{(\lambda)}_\nu$ is a $(-1)$-curve. 
Thus there exists a contraction 
$\overline{p}\colon\widetilde{W}\to\overline{W}$ of each $\widetilde{E}^{(\lambda)}_\nu$. 
Denote by $\overline{C}$ the strict transform of $p^{-1}(C)$. 
We can deduce from the same argument as in the proof of Proposition \ref{prop:hodge} that 
$H^1(\overline{C}, \mathcal{O}_{\overline{C}})\cong H^1(\widetilde{D}, \mathcal{O}_{\widetilde{D}})$. 
Thus 
$\overline{p}^*u_n(\overline{C}, \overline{W})=u_n(\widetilde{D}, \widetilde{W})=0$ implies that $u_n(\overline{C}, \overline{W})=0$. 
Since the pair $(\overline{W}, \overline{C})$ is of infinite type, we can apply \cite[Theorem 3]{U83} to show that $\mathcal{O}_{\overline{W}}(\overline{C})$ is a flat line bundle. 
As it follows that the line bundle $\mathcal{O}_{\widetilde{W}}(\widetilde{D})$ is also flat, we can take a nowhere vanishing section $F\in H^0\left(\widetilde{W},\ p^*N^{-1}\otimes \mathcal{O}_{\widetilde{W}}(\widetilde{D})\right)$, 
where $N$ is a flat line bundle on $W$ such that $N|_{{\rm Supp}\,D}=N_{D/W}$ holds. 
Consider the section 
\[
F\otimes i^*F\otimes (i^{2})^*F\otimes \dots (i^{a-1})^*F\in H^0\left(\widetilde{W},\ p^*N^{-a}\otimes \mathcal{O}_{\widetilde{W}}(a\widetilde{D})\right)
=H^0\left(\widetilde{W}, p^*\left(N^{-1}\otimes \mathcal{O}_W(D)\right)\right), 
\]
where $i$ is a generator of the group ${\rm Aut}\,(\widetilde{W}/W)$ of all deck transformation of $p$. 
As it clearly holds that the section above is nowhere vanishing and ${\rm Aut}\,(\widetilde{W}/W)$-invariant, 
it can be realized as the pull-back of a nowhere vanishing holomorphic gobal section of the line bundle $N^{-1}\otimes \mathcal{O}_W(D)$, which induces the isomorphism $N\cong\mathcal{O}_W(D)$. 
\end{proof}

\begin{proposition}\label{thm:SNCuedath_2}
Let $(X, D), W, L$ be as above. 
Assume that each $E_\nu$ is a rational curve, $N_{D/W}=\mathcal{O}_{{\rm Supp}\,D}$, 
and that the pair $(\widetilde{W}, \widetilde{D})$ is of type $n<\infty$. 
Then the following holds: \\
$(i)$ For each real number $\lambda>1$, There exists  a neighborhood $V$ of ${\rm Supp}\,D$ and a strongly psh function $\Phi_\lambda\colon V\setminus {\rm Supp}\,D\to \mathbb{R}$ such that $\Phi_\lambda(p)\to\infty$ and $\Phi_\lambda(p)=O(d(p, {\rm Supp}\,D)^{-\lambda n/a})$ hold as $p\to {\rm Supp}\,D$, where $d(p, {\rm Supp}\,D)$ is the distance from $p$ to ${\rm Supp}\,D$ calculated by using a local Euclidean metric on a neighborhood of a point of $C$ in $V$. \\
$(ii)$ Let $V$ be a neighborhood of ${\rm Supp}\,D$ in $X$, $\Psi$ be a psh function defined on$V\setminus {\rm Supp}\,D$. 
If there exists a real number $0<\lambda<1$ such that $\Psi(p)=O(d(p, {\rm Supp}\,D)^{-\lambda n/a})$ as $p\to {\rm Supp}\,D$, then there exists a neighborhood $V_0$ of ${\rm Supp}\,D$ in $V$ such that $\Psi|_{V_0\setminus {\rm Supp}\,D}$ is a constant function. 
\end{proposition}

\begin{proof}
Let $\overline{p}\colon\widetilde{W}\to\overline{W}$ and $\overline{C}$ be those in the proof of Proposition \ref{thm:SNCuedath_1}. 
In this case, we can deduce that the pair $(\overline{W}, \overline{C})$ is of type $n<\infty$ by the same argument. 
Thus we can apply \cite[Theorem 1]{U83} to show that, for each $\lambda>1$, by shrinking $\overline{W}$ if necessary, 
there exists a strongly psh function $\Phi'$ on $\overline{W}\setminus\overline{C}$ with 
$\Phi'(p)=O(|f_{\overline{C}}(p)|^{-\alpha})$ as $p\to\overline{C}$, 
where $f_{\overline{C}}$ is a local defining function of $\overline{C}$. 
Then the function 
$\sum_{\ell=1}^a(i^{\ell})^*(\overline{p}^*\Phi')$
can be realized as a pull-back of a strongly psh function $\Phi$ by $p$, where $i$ is a generator of ${\rm Aut}\,(\widetilde{W}/W)$. 
The assertion $(i)$ can be proved by considering this $\Phi$. 
The assertion $(ii)$ followed from the same argument as in the proof of \cite[Theorem 2]{U83} for the functions $\Phi$ and $p^*\Psi$. 
\end{proof}

\subsection{Proof of Theorem \ref{thm:maincor}}

In this subsection, we prove Theorem \ref{thm:maincor}. 

When $N_{C/X}\in\mathcal{E}_1(C)$ holds, it follows from Corollary \ref{cor:main_P1} that 
there exists a neighborhood $V$ of $C$ such that $\mathcal{O}_V(C)$ is flat. 
Thus, by the argument as in \cite[Corollary 3.5]{K1}, we can show that Theorem 
\ref{thm:maincor} $(i)$ holds. 
When $N_{C/X}\not\in\mathcal{P}_0(C)$, we can use \ref{thm:main_2} 
and run the same argument as in the proof of \cite[Theorem 1.1]{K2}, 
which proves Theorem \ref{thm:maincor} $(ii)$. 
\qed

\subsection{Proof of Theorem \ref{thm:9ptbup}}

In this subsection, we prove Theorem \ref{thm:9ptbup}. 
Let $p_1, p_2, \dots, p_9\in\mathbb{P}^2$ be $9$ points different from each other and denote by $\pi\colon X\to\mathbb{P}^2$ the blow-up at $\{p_j\}_{j=1}^9$. 
Take a curve $C_0\subset\mathbb{P}^2$ of degree $3$ which includes $\{p_j\}_{j=1}^9$ as elements. 

\begin{proposition}\label{prop:9ptbup_all}
\ \\
$(i)$ If $K_X^{-1}$ is not nef, the positive part of $K_X^{-1}$ is semi-ample. \\
$(ii)$ If $K_X^{-1}$ is nef, $C_0$ can be taken as a curve with $\{p_j\}_{j=1}^9\cap C_{\rm sing}=\emptyset$ and satisfying one of the following seven conditions: 
\begin{description}
  \setlength{\parskip}{0cm} 
  \setlength{\itemsep}{0cm} 
\item[$(7.10.1)$] $C_0$ is a non-singular elliptic curve. 
\item[$(7.10.2)$] $C_0$ is a rational curve with a node. 
\item[$(7.10.3)$] $C_0$ is a rational curve with a cusp. 
\item[$(7.10.4)$] $C_0$ has only nodes, and it has two irreducible components: $\ell_0$ with degree $1$ and $\ell_1$ with degree $2$. It also hold that $\#(\{p_j\}_{j=1}^9\cap \ell_0)=3$ and that 
$\#(\{p_j\}_{j=1}^9\cap \ell_1)=6$. 
\item[$(7.10.5)$] $C_0$ has two irreducible components: $\ell_0$ with degree $1$ and $\ell_1$ with degree $2$. It also hold that $\#(\ell_0\cap\ell_1)=1$, $\#(\{p_j\}_{j=1}^9\cap \ell_0)=3$, and that 
$\#(\{p_j\}_{j=1}^9\cap \ell_1)=6$. 
\item[$(7.10.6)$] $C_0$ has only nodes, and it has three irreducible components $\ell_0, \ell_1, \ell_2$ with degree $1$. It also holds that $\#(\{p_j\}_{j=1}^9\cap \ell_j)=3$ for each $j=0, 1, 2$. 
\item[$(7.10.7)$] $C_0$ has three irreducible components $\ell_0, \ell_1, \ell_2$ with degree $1$. It also hold that $\#(\ell_0\cap\ell_1\cap\ell_2)=1$ and that $\#(\{p_j\}_{j=1}^9\cap \ell_j)=3$ for each $j=0, 1, 2$. 
\end{description}
\end{proposition}

\begin{proof}
First we show the assertion $(i)$. 
Let $P$ be the nef part of Zariski decomposition of $K_X^{-1}$ and denote by $N$ the negative part: $K_X^{-1}=P\otimes\mathcal{O}_X(N)$. 
When $K_X^{-1}$ is not nef, then $(K_X^{-1}. P)=-((P+N). N)=-(N^2)>0$ holds. 
Thus we can apply \cite[Lemma 3.1]{LT} and conclude that $P$ is semi-ample. 
$(ii)$ can be shown by elemental arguments can be shown by case-by-case argument. 
\end{proof}

It follows from the above proposition that, 
in order to prove Theorem \ref{thm:9ptbup}, it is sufficient to show one of the assertions $(i), (ii), \dots, (iv)$ holds by assuming each condition $(7.10.1), (7.10.2), \dotsm (7.10.7)$ in Proposition \ref{prop:9ptbup_all} $(ii)$. 
In the following, we denote by $C$ the strict transform of $C_0$. 

\subsubsection{case $(7.10.1)$}

In this case, the pair $(C, X)$ is of infinite (see \cite[Lemma 6.2]{N}). 
Thus we can apply \cite[Theorem 3]{U83} to show that the assertion $(i), (ii)$, or $(iv)$ in Theorem \ref{thm:9ptbup} holds (see \cite{Br}). 

\subsubsection{case $(7.10.2)$}

When $N_{C/X}\not\in\mathcal{P}_0(C)$, we can 
use Theorem \ref{thm:maincor} $(ii)$ 
to show the assertion Theorem \ref{thm:9ptbup} $(iii)$ holds. 
Thus all we have to do is to consider the case when $N_{C/X}\in\mathcal{P}_0(C)$. 
If $N_{C/X}\in\mathcal{E}_0(C)$ holds, it follows from Lemma \ref{lem:torsion_9ptbup} that the assertion Theorem \ref{thm:9ptbup} $(i)$ holds. 
If $N_{C/X}\in\mathcal{E}_1(C)$ holds, we can apply Theorem \ref{thm:maincor} $(i)$
and show that the assertion Theorem \ref{thm:9ptbup} $(ii)$ holds. 
The rest case is Theorem \ref{thm:9ptbup} $(iv)$. 

\subsubsection{case $(7.10.3)$}

Let us denote by $q$ the cuspidal point of $C$. 
Let $\pi_1\colon S_1\to X$ be the blow-up at $q$, 
$\pi_2\colon S_2\to S_1$を${\rm Supp}\,\pi_1^*C$ the blow-up at the singular point, 
and $\pi_3\colon S\to S_2$を${\rm Supp}\,\pi_2^*\pi_1^*C$ be the blow-up at the singular point. 
Denote by $\pi$ the map $\pi_3\circ\pi_2\circ\pi_1\colon S\to S_0$. 
Let us denote by $C_1\subset S$ the strict transform of the exceptional curve of $\pi_3$, 
by $E_1\subset S$ the strict transform of the exceptional curve of $\pi_2$, 
by $E_2$ the strict transform of the exceptional curve of $\pi_1$, 
and by $E_3\subset S$ the strict transform of $C_0$. 
Then the divisor $D:=\pi^*C$ can be decomposed as $D=6C_1+3E_1+2E_2+E_3$. 

Consider $p\colon\widetilde{W}\to S$ and $\widetilde{D}$ as in Lemma \ref{lem:existence_Wtild}. 
By a calculation as in Proposition \ref{prop:hodge}, it follows that $H^1({\rm Supp}\,D, \mathcal{O}_{{\rm Supp}\,D})=0$ and thus $\mathcal{P}({\rm Supp}\,D)=\{\mathcal{O}_{{\rm Supp}\,D}\}$. 
Therefore it turns out that $N_{D/S}=\mathcal{O}_{{\rm Supp}\,D}$. 
When the pair $(\widetilde{D}, \widetilde{W})$ is of infinite type, 
it follows from Proposition \ref{thm:SNCuedath_1} that there exists a neighborhood $W$ of ${\rm Supp}\,D$ such that $\mathcal{O}_W(D)=\mathcal{O}_W$. 
Thus we obtain that ${\rm dim}\,H^0(C, \mathcal{O}_X(C)|_C)=1$. 
Considering the exact sequence $H^0(X, \mathcal{O}_X(C))\to H^0(C, \mathcal{O}_X(C)|_C)\to H^1(X, \mathcal{O}_X)=0$ induced by 
$0\to\mathcal{O}_X\to\mathcal{O}_X(C)\to\mathcal{O}_X(C)\otimes\mathcal{O}_X/\mathcal{O}_X(-C)\to 0$, 
we can conclude that the assertion Theorem \ref{thm:9ptbup} $(i)$ holds in this case. 
When the pair $(\widetilde{D}, \widetilde{W})$ is of finite type, 
we can use Proposition \ref{thm:SNCuedath_2} and run the argument as in the proof of \cite[Theorem 1.1]{K2} to conclude that the assertion Theorem \ref{thm:9ptbup} $(iii)$ holds. 

\subsubsection{case $(7.10.4)$}

In this case, we can show Theorem \ref{thm:9ptbup} by the same argument as case $(7.10.2)$. 

\subsubsection{case $(7.10.5)$}

In this case, we can show Theorem \ref{thm:9ptbup} by the same argument as case $(7.10.3)$. 

\subsubsection{case $(7.10.6)$}

In this case, we can show Theorem \ref{thm:9ptbup} by the same argument as case $(7.10.2)$. 

\subsubsection{case $(7.10.7)$}

In this case, we can show Theorem \ref{thm:9ptbup} by the same argument as case $(7.10.3)$. 
\qed


\end{document}